\documentclass{article}
\usepackage{amsmath,amssymb,amsthm,graphicx,epsfig,float,url}
\usepackage[colorlinks=true,citecolor={Plum},linkcolor={Periwinkle}]{hyperref}
\usepackage{pdfsync}
\usepackage[usenames, dvipsnames]{xcolor}
\usepackage{tikz}
\usepackage{subfig}
\usepackage{bbm}
\usepackage{stmaryrd}
\usepackage{amssymb}
\usepackage{mathrsfs}  
\usepackage{caption}
\usepackage{float}
\usepackage{esint}

\usepackage{dsfont}
\usepackage[makeroom]{cancel}
\usetikzlibrary{patterns}
\newcommand{\R}{\textnormal{I\kern-0.21emR}}
\newcommand{\N}{\textnormal{I\kern-0.21emN}}

\renewcommand{\geq}{\geqslant}
\renewcommand{\leq}{\leqslant}

\def\H{{\mathcal H^{d-1}}}
\def\e{{\varepsilon}}

\def\Vol{{\operatorname{Vol}}}

\allowdisplaybreaks

\def\YYint#1#2#3{{\setbox0=\hbox{$#1{#2#3}{\iint}$}
    \vcenter{\hbox{$#2#3$}}\kern-.51\wd0}}
 
\usepackage[dvipsnames]{xcolor}

\newtheorem*{theorem*}{Theorem}

\newtheorem{theorem}{Theorem}

\newtheorem{material}{material}
\newtheorem{proposition}[material]{Proposition}

\newtheorem{definition}[material]{Definition}
\newtheorem{lemma}[material]{Lemma}

\newtheorem{remark}[material]{Remark}

\def\O{{\Omega}}
\def\b{{\beta^*}}
\def\n{{\nabla}}

\def\p{{\varphi}}

\def\ell{{\operatorname{ell}}}
\usepackage{xargs}
 \usepackage[colorinlistoftodos,textsize=small]{todonotes}
 \newcommandx{\christian}[2][1=]{\todo[linecolor=red,backgroundcolor=red!25,bordercolor=red,#1]{#2}}
 \newcommandx{\laura}[2][1=]{\todo[linecolor=blue,backgroundcolor=blue!25,bordercolor=blue,#1]{#2}}
 \newcommandx{\info}[2][1=]{\todo[linecolor=green,backgroundcolor=green!25,bordercolor=green,#1]{#2}}
 \newcommandx{\improvement}[2][1=]{\todo[linecolor=yellow,backgroundcolor=yellow!25,bordercolor=yellow,#1]{#2}}
 
  \newcommandx{\biblio}[2][1=]{\todo[linecolor=blue,backgroundcolor=magenta!25,bordercolor=blue,#1]{#2}}

 \numberwithin{equation}{section}

\begin{document}
\title{Qualitative analysis of optimisation problems with respect to non-constant Robin coefficients}


\author{Idriss Mazari\footnote{CEREMADE, UMR CNRS 7534, Universit\'e Paris-Dauphine, Universit\'e PSL, Place du Mar\'echal De Lattre De Tassigny, 75775 Paris cedex 16, France, (\texttt{mazari@ceremade.dauphine.fr})}
\and Yannick Privat\footnote{IRMA, Universit\'e de Strasbourg, CNRS UMR 7501, Inria, 7 rue Ren\'e Descartes, 67084 Strasbourg, France ({\tt yannick.privat@unistra.fr}).}\textsuperscript{~~}\footnote{Institut Universitaire de France (IUF).} }
\date{\today}

\maketitle

\begin{abstract}
Following recent interest in the qualitative analysis of some optimal control and shape optimisation problems, we provide in this article a detailed study of the optimisation of Robin boundary conditions in PDE constrained calculus of variations. Our main model consists of an elliptic PDE of the form $-\Delta u_\beta=f(x,u_\beta)$ endowed with the Robin boundary conditions $\partial_\nu u_\beta+\beta(x)u_\beta=0$. The optimisation variable is the function $\beta$, which is assumed to take values between 0 and 1 and to have a fixed integral. Two types of criteria are under consideration: the first one is non-energetic criteria. In other words, we aim at optimising functionals of the form $\mathcal J(\beta)=\int_{\O\text{ or }\partial \O}j(u_\beta)$. We prove that, depending on the monotonicity of the function $j$, the optimisers may be of \emph{bang-bang} type (in other words, the optimisers write $\mathds 1_\Gamma$ for some measurable subset $\Gamma$ of $\partial \O$) or, on the contrary, that they may only take values strictly between 0 and 1. This has consequence for a related shape optimisation problem, in which one tries to find where on the boundary Neumann ($\partial_\nu u=0$ ) and  constant Robin conditions ($\partial_\nu u+u=0$) should be placed in order to optimise criteria.  The proofs for this first case rely on new fine oscillatory techniques, used in combination with optimality conditions. We then investigate the case of compliance-type functionals. For such energetic functionals, we give an in-depth analysis and even some explicit characterisation of optimal $\beta^*$.
\end{abstract}

\noindent\textbf{Keywords:} Elliptic boundary values problems, Robin boundary conditions, Calculus of variations, Shape optimisation, Bilinear optimal control problems, Qualitative analysis of optimisation problems

\medskip

\noindent\textbf{AMS classification:} 49J15, 49Q10.

\paragraph{Acknowledgment.} This work was  partially funded by the French ANR Project ANR-18-CE40-0013 - SHAPO on Shape Optimization and by the Project "Analysis and simulation of optimal shapes - Application to lifesciences" of the Paris City Hall.


\section{Introduction}
\subsection{Scope of the article, informal presentation of the problem}

\subsubsection{Informal presentation of the problem}
The goal of this article is to provide a theoretical analysis of a class of PDE constrained optimisation problems which arise in many fields (for instance, in automation, in physics or in mathematical biology), and in which  the aim is to minimise or maximise a certain criteria by acting on the coefficients of the Robin boundary conditions. 

More specifically, we are working with \emph{heterogeneous Robin boundary conditions}, in the following sense: the state equation of the phenomenon is supplemented, on the boundary of the domain, with a condition of the form 
$$\frac{\partial u}{\partial \nu}(x)+\beta(x)u(x)=0,$$ where $\beta$ is a non-negative function on the boundary. Our goal is to optimise certain criteria with respect to $\beta$, under some natural constraints.

\paragraph{Context} We provide bibliographical references in section \ref{Se:Bib} of this introduction, but let us point out that such problems have been the topic of a wide research activity.  For instance, in \cite{BUCUR2017451,Bucur2017}, several aspects of the optimisation of the natural energy of the underlying PDE or of some eigenvalues were tackled. Similarly, this type of question is very natural in the context of thermal insulation. In this case, a relevant query is to find the best way to coat a domain with an insulant in order to optimise certain criteria. This is the point of view chosen, for example, in \cite{BUCUR2017451,Bucur2017,Pietra2020}. Other authors have studied this problem in parabolic models, with applications to fluid dynamics \cite{MR3009728}, or in hyperbolic problems \cite{MR1623313}. Let us finally mention that, in many of the aforementioned cases, the functionals to be optimised either derive from the natural energy of the PDE or are of "tracking-type" (\emph{i.e.} the aim is to minimise the distance of the state to a reference state). However, many relevant optimisation problems do not fall in either category. This is for example the case in spatial ecology. One may consider, following \cite{BaiHeLi,LouNagaharaYanagida,MNP2021,MNP,MRBSIAP,NagaharaYanagida}, the problem of maximising the population size in logistic models: how should one design the boundary of a domain in order to optimise the total population inside the domain? Although most of our analysis is, in the main proofs of this article, detailed in the case of linear models, we also provide in section \ref{Se:NL} some extension to non-linear models.

We consider and analyse fairly general functionals, with a strong emphasis on the \emph{qualitative properties} of optimisers. As we shall see, these properties are closely related to \emph{existence results for shape optimisation problems}. To carry out our proofs, we introduce an oscillatory method which, although reminiscent of the one we introduced recently in \cite{MNP2021}, requires fine tuning to obtain our results. 

One of our main contributions is the analysis of the influence of the type of functionals we wish to optimise (for instance,  the \emph{monotonicity of the functional} is crucial in the forthcoming analysis) on qualitative results. Let us mention that we consider two main types of functionals: energetic ones (in other words, functionals that are equal, up to a multiplicative constant, to the natural energy of the state equation), in which we may achieve an explicit characterisation of optimisers, and non-energetic ones. The latter case exhibits very different qualitative features; this is the main topic under consideration here.

Furthermore, let us underline that, from a mathematical perspective, our contributions can be read through the lens of \emph{bilinear optimal control problems} set on the boundary. In this setting, we see $\beta$ as the control. Likewise, such bilinear optimal control problems have been an active topic of research in the past years, and are not yet fully understood. We refer once again to section \ref{Se:Bib} of the introduction for further discussion of recent works in this field.



 \paragraph{Paradigmatic formulation of our problem and motivation}
The most general version of our problem problem reads as follows: let $\O$ be a regular enough domain in $\R^d$ and let, for any measurable subset $\Gamma\subset \partial \O$, $u^\Gamma_\alpha$ be the unique solution of 
\begin{equation}\label{Eq:MainIntro}\begin{cases}
 -\nabla \cdot(A(x)\nabla u^\Gamma_\alpha(x))=g_0(x,u^\Gamma_\alpha(x))& x\in\O 
 \\\partial_\nu u^\Gamma_\alpha(x)+\alpha \mathds 1_\Gamma(x) u^\Gamma_\alpha(x)=0& x\in\partial\O. 
 \end{cases}
 \end{equation}
 In this formulation, $\alpha>0$ is a fixed parameter, $A$ is a matrix assumed to be uniformly elliptic and $g_0$ is a given non-linearity.
 
   A possible interpretation of this equation is that it is an approximation of a mixed problem of the type
 \begin{equation}\label{Eq:MainIntro2}\begin{cases}
 -\nabla \cdot(A(x)\nabla v^\Gamma(x))=g_0(x,v^\Gamma(x))& x \in \O, 
 \\\partial_\nu v^\Gamma(x)=0& x \in \partial \O\backslash \Gamma, 
 \\ v^\Gamma(x)=0&x \in \Gamma.
 \end{cases}
 \end{equation}
 Indeed, under adequate assumptions, the solution to \eqref{Eq:MainIntro} converges in some sense to the one of \eqref{Eq:MainIntro2} as $\alpha\to +\infty$. We refer to Appendix~\ref{append_vGamma} for additional explanation in the case where $g_0=g_0(x)$. 
 
As a consequence,  optimising criteria involving $u^\Gamma_\alpha$ with respect to $\Gamma$ is closely related to optimising criteria  involving $v^\Gamma$ with respect to $\Gamma$. 
\begin{remark}[Comment on the methods used] Let us stress the following fact: while it is plausible that the optimisers of the problems involving $u_\alpha^\Gamma$ and of the problems involving $v^\Gamma$ have the same qualitative features, we believe that the tools necessary in order to analyse them are fundamentally different, as the proper convergence for the mixed Neumann-Dirichlet problem should be the $\Gamma$-convergence of sets while, for the problems involving $u_\alpha^\Gamma$, the relevant topology is rather the weak $L^\infty-*$ one on the compactification of $\{\mathds 1_\Gamma\,, \H(\Gamma)=V_0\}$.  We  refer to section \ref{Se:Tech} for additional comments.\end{remark}

The typical problems we consider in this paper are
 \begin{equation}\label{Eq:PvIntro}
 \underset{\substack{\Gamma\subset \partial\Omega\\ \H\left(\Gamma\right)=V_0}}{\sup/\inf}\int_\O j\left(u^\Gamma_\alpha\right)\quad \text{ or }\quad \underset{\substack{\Gamma\subset \partial\Omega\\ \H\left(\Gamma\right)=V_0}}{\sup/\inf}\int_{\partial \O} j\left(u^\Gamma_\alpha\right)
 \end{equation}
 where $\H(\cdot)$ stands for the $(d-1)$-Hausdorff dimensional measure of $\Gamma$ and $j$ is a smooth non-linearity. Stated as such, these problems are shape optimisation problems. We consider a relaxed version of this problem, where the term $\alpha \mathds 1_\Gamma$ in \eqref{Eq:MainIntro} is replaced with a function $\beta\in L^\infty(\partial \O)$ satisfying certain constraints; we explain later in this introduction how qualitative properties for the optimisation with respect to $\beta$ are translated to (non-)existence results for the initial shape optimisation problem. 
 
 We also consider, for the sake of completeness, compliance-type problems: namely, the goal for this class of problems is to solve
 \begin{equation}\label{Eq:PvIntro2}
 \underset{\substack{\Gamma\subset \partial\Omega\\ \H\left(\Gamma\right)=V_0}}{\sup/\inf}\int_\O fu_\alpha^\Gamma.\end{equation}
  
Since all the results we establish hereafter are valid whatever the value of $\alpha>0$, we may without loss of generality fix $\alpha=1$ in this formulation. We also underline that we first work with linear models and cover non-linear models in section \ref{Se:NL}.

\subsubsection{Plan of the introduction} The introduction of this paper is long, and we thus take the liberty to give a plan to ease the reading.
Subsection \ref{Su:Cons} contains the presentation of the state equation, as well as the definitions of the two types of functionals considered. In particular, it concludes with a presentation of the focal point of our analysis for non-energetic functionals, the bang-bang property. In subsection \ref{Se:Relaxation}, we motivate the analysis of this property by linking it to existence properties for some shape optimisation problems. in sections \ref{Se:Results}-\ref{Se:Results2}, we give our main theorems in the linear case, first stating the ones dealing with non-energetic criteria, second presenting the ones relevant for energetic criteria. Section \ref{Se:NL} contains the results for non-linear models. Section \ref{Se:Tech} is devoted to the technical context of our proofs. The introduction concludes with section \ref{Se:Bib}, which contains a discussion of the relevant references. 

 \subsubsection{Problem under consideration in this article}\label{Su:Cons}
 \paragraph{Relaxation of the problem and admissible class of coefficients $\boldsymbol{\beta}$.} We have mentioned we would consider a relaxed version of \eqref{Eq:PvIntro}-\eqref{Eq:PvIntro2}. In order to make the above statement about the relaxation of the problem precise,
 we define, for a fixed $V_0\in (0;\mathcal H^{d-1}(\partial \O))$, the set $\mathcal B(\partial \O)$ as
 \begin{equation}\label{Eq:AdmRobin}
 \mathcal B(\partial \O):=\left\{\beta\in L^\infty(\partial \O):\, 0\leq \beta\leq 1, \int_{\partial \O}\beta=V_0\right\}.\end{equation} 
 This set corresponds corresponds to the closure of the set $\{\mathds 1_\Gamma, \H\left(\Gamma\right)=V_0\}$ for the weak-star topology of $L^\infty(\partial\Omega)$ \cite[Proposition~7.2.17]{henrot-pierre}. The set $\mathcal B(\partial \O)$ is the admissible class we consider throughout this paper. The link between existence properties for the shape optimisation problems of type \eqref{Eq:PvIntro} and the so-called\emph{ bang-bang property} for optimisation problems set in $\mathcal B(\partial \O)$ is investigated in section \ref{Se:Relaxation}.
 
 \paragraph{State equation} For the sake of simplicity, we first focus in this paper on a simpler version of \eqref{Eq:MainIntro2}.  This allows us to not dwell on existence and regularity issues, and we thus hope to provide a clear description of the type of arguments involved in the proofs of our results. We refer to section \ref{Se:NL} for non-linear models.
 
 Henceforth, $\O$ is a fixed open bounded connected subset of $\R^d$ with a $\mathscr C^2$ boundary, and $f\in L^\infty(\O)$ is a fixed source term. We further assume that
 \begin{equation}\label{Eq:Hypf}\tag{$\bold H_f$}
 f\text{ does not vanish identically and }f\geq 0\text{ a.e. in }\O.
 \end{equation}
For any $\beta\in \mathcal B(\partial \O)$, we denote by $u_\beta$ the unique solution of the equation
\begin{equation}\label{Eq:MainRobin}\tag{$\bold E_\beta$}\begin{cases}
-\Delta u_\beta=f&\text{ in }\O, 
\\ \frac{\partial u_\beta}{\partial \nu}+\beta u_\beta=0&\text{ on }\partial\O.
\end{cases}
\end{equation}
Alternatively, $u_\beta$ is the unique minimiser in $W^{1,2}(\O)$ of the energy functional 
\begin{equation}\label{Eq:Energy}
\mathcal E_\beta:W^{1,2}(\O)\ni u\mapsto \frac12\int_\O |\n u|^2+\frac12\int_{\partial \O}\beta u^2-\int_\O fu.\end{equation} 
As a consequence of the strong maximum principle, for any $\beta\in \mathcal B(\partial \O)$ and any $f\in L^\infty(\O)$ satisfying \eqref{Eq:Hypf}, one has
$$
\inf_{\overline \O}u_\beta>0.
$$
\paragraph{First type of functional: energetic functionals}
Two natural optimisation problems that stem from \eqref{Eq:MainRobin} are the problems of maximising and minimising the compliance. In other words, we shall tackle the problems
\begin{equation*}
\underset{\beta \in \mathcal B(\partial \O)}{\max/\min}\int_\O fu_\beta=\mathcal F(\beta) \textnormal{ \textsf{(energetic criteria)}}.
\end{equation*}
What is notable here is that the functional $\mathcal F$ rewrites naturally using the energy defined in \eqref{Eq:Energy}, whence the wording "energetic" to describe such criteria. Straightforward computations indeed lead to 
\begin{equation*}
\mathcal F(\beta)=-2\mathcal E_\beta(u_\beta)=-2\min_{u\in W^{1,2}(\O)}\mathcal E_\beta(u).\end{equation*}This alternative formulation enables us to obtain a finer description of optimisers. Since this is not the central point of this paper, we state the relevant results last, in Theorem \ref{Theo:minmaxNRJ}.

\paragraph{Second type of functional: non-energetic criteria}
We want to consider boundary and interior cost functionals. For both these criteria, we consider a fixed non-linearity $j:\R\to \R$. As we will see, \emph{the monotonicity of $j$ plays a crucial role in the qualitative analysis of optimisers}. Thus, we choose $j$ to be monotonous. Since we are dealing both with minimisation and maximisation problems, we may take $j$ to be increasing. Overall, we  assume that 
\begin{equation}\label{Eq:Hypj}\tag{$\bold H_j$}
j\in \mathscr C^2(\R_+)\text{ and } j'>0\text{ on }\R_+^*.
\end{equation} Using this non-linearity we define two functionals:
\begin{equation}\label{Eq:PvDef}
\begin{array}{ll}
\displaystyle \mathcal J_{\partial \O}:\mathcal B(\partial \O)\ni\beta\mapsto \int_{\partial \O}j(u_\beta)& \textnormal{ \textsf{(boundary criterion)}},\\ \\
\displaystyle \mathcal J_{ \O}:\mathcal B(\partial \O)\ni\beta\mapsto \int_{ \O}j(u_\beta) & \textnormal{ \textsf{(distributed criterion)}}, 
\end{array}
\end{equation}
Our focus is on the optimisation problems
\begin{equation}\label{Eq:2}
\underset{\beta \in \mathcal B(\partial \O)}{\sup/\inf}\int_{\partial \O}j(u_\beta),
\end{equation}
where $u_\beta$ solves \eqref{Eq:MainRobin}. 
We refer again to section \ref{Se:Bib} of this introduction for a discussion of the history of these problems. Several features of the optimisers are relevant in such queries, among them the so-called bang-bang property: \emph{do optimisers $\beta^*$ write $\mathds 1_\Gamma$ for some subset $\Gamma$ of $\partial \O$}? To justify why this is a relevant question, let us now discuss briefly the relationship between this bang-bang property and the existence of optimal shapes for the related shape optimisation problem.

\subsubsection{Relationship between existence properties for the shape optimisation problem and the bang-bang property for the relaxed formulation.}\label{Se:Relaxation} Let us now explain a bit more in details the link between the initial shape optimisation problems 
\begin{equation}\label{Eq:1}
\underset{\substack{\Gamma\subset \partial\Omega\\ \H\left(\Gamma\right)=V_0}}{\sup/\inf} \int_{\partial \O}j\left(w^\Gamma\right),
\end{equation} 
where $w^\Gamma$ solves 
\begin{equation}\begin{cases}
-\Delta w^{\Gamma}=f&\text{ in }\O,\\ 
\frac{\partial w^\Gamma}{\partial \nu}+\mathds 1_\Gamma w^\Gamma=0&\text{ on }\partial \O,\end{cases}\end{equation} and the relaxed problem \eqref{Eq:2}.

The first thing that should be noted is that \eqref{Eq:2} has a solution $\beta^*$. We refer to Lemma \ref{Le:Exist} below and indicate that this follows from the direct method in the calculus of variations: the weak $L^\infty-*$ compactness of $\mathcal B(\partial \O)$ and the continuity for the $L^\infty-*$ topology of the functionals suffice to obtain this result. Obtaining such an existence property for \eqref{Eq:1} is much harder.

 However, since $\mathcal B(\partial \O)$ is the compactification of the set $\{\mathds 1_\Gamma \,, \H\left(\Gamma\right)=V_0\}$ it follows that for every $\beta\in \mathcal B(\partial \O)$ there exists a sequence $\{\Gamma_k\}_{k\in \N}$ of measurable subsets of $\partial \O$ with Hausdorff measure $V_0$ such that 
 $$\mathds 1_{\Gamma_k}\underset{k\to \infty}\to \beta\text{ in the weak $L^\infty-*$ topology}.$$ Since $\mathcal J_{\partial \O}$ is continuous for this topology, we obtain
$$
\mathcal J_{\partial \O}\left(\mathds 1_{\Gamma_k}\right)\xrightarrow[k\to +\infty]{} \mathcal J_{\partial\O}( \beta).
$$ 
The set $\{\mathds 1_\Gamma, \mathcal H^{d-1}(\Gamma)=V_0\}$ corresponds exactly the set of extreme points of the admissible set $\mathcal B(\partial \O)$. We call its elements \emph{bang-bang functions}.

With these informations it is easy to obtain the following proposition/definition describing the relationships between the shape optimisation problem~\eqref{Eq:1} and the bilinear optimal control problem \eqref{Eq:2}.
\begin{definition}\label{De:Link}
\begin{enumerate}
\item Problem~\eqref{Eq:1} has a solution if, and only if there exists a bang-bang solution $\beta^*$ of Problem~\eqref{Eq:2}. In this case, we say that Problem~\eqref{Eq:2} satisfies the bang-bang property.
\item Alternatively, Problem~\eqref{Eq:1} does not have a solution if, and only if any solution $\beta^*$ of \eqref{Eq:2} satisfies $\Vol\left(\{0<\beta^*<1\}\right)>0$. In this case, we say that Problem~\eqref{Eq:1} enjoys a relaxation property.
\end{enumerate}
\end{definition}

 \subsection{First case: boundary and distributed criteria}\label{Se:Results}
 \subsubsection{Existence results and bang-bang property for maximisation problems }

\begin{center}
\fbox{\textsf{Boundary criteria}}
\end{center}
We first tackle the maximisation problem
\begin{equation}\tag{$\bold P_{\max,\partial \O,\mathcal B}$}\label{Eq:MaxBound}
\max_{\beta\in \mathcal B(\partial \O)}\int_{\partial \O} j(u_\beta),
\end{equation} 
where $u_\beta$ denotes the unique solution to \eqref{Eq:MainRobin}.
We refer to Lemma~\ref{Le:Exist}  below for the existence of optimal profiles. We also state the related shape optimisation problem
\begin{equation}\label{Eq:MaxBoundSo}\tag{$\bold P_{\max,\partial \O,\Sigma}$}
\sup_{\substack{\Gamma\subset \partial \O\\ \mathcal H^{d-1}(\Gamma )=V_0}}\int_{\partial \O} j\left(w^\Gamma\right).
\end{equation} 
Our first result states that maximisers of \eqref{Eq:MaxBound} satisfy the bang-bang property. Following the discussion of Section \ref{Se:Relaxation}, the shape optimisation problem \eqref{Eq:MaxBoundSo} has a solution.

\begin{theorem}\label{Th:BgbgBound}
Let $\O$ be a bounded open set of $\R^d$ such that $\partial\O$ is $\mathscr{C}^2$. 
Assume $f$ satisfies \eqref{Eq:Hypf} and $j$ satisfies \eqref{Eq:Hypj}.
Any solution $\beta^*$ of the optimisation problem \eqref{Eq:MaxBound} is bang-bang: there exists $\Gamma^*\subset \partial \O$ such that $ \beta^*=\mathds 1_{\Gamma^*}$. As a consequence, the shape optimisation problem \eqref{Eq:MaxBoundSo} has a solution.
\end{theorem}

The proof of this Theorem is one of the central points of this paper. It is carried out using a high frequency analysis of the second order derivative of the functional. While this type of results is usually proved using convexity or concavity arguments, the fact that the problem is not energetic and that we are considering a bilinear control problems \emph{a priori} prohibits obtaining a convexity property for the functional $\mathcal J_{\partial \O}$. We refer to section \ref{Se:Tech} for a discussion of the method.

\begin{center}
\fbox{\textsf{Distributed criteria}}
\end{center}
Although we decided to start with boundary criteria as, to the best of our knowledge, they have received less attention in the literature, our methods naturally extend to the case of distributed criteria.  In this case, the optimisation problem is
\begin{equation}\tag{$\bold P_{\max, \O,\mathcal B}$}\label{Eq:MaxDist}
\max_{\beta\in \mathcal B(\partial \O)}\int_{ \O} j(u_\beta),
\end{equation} 
where $u_\beta$ denotes the unique solution to \eqref{Eq:MainRobin}, and the related shape optimisation problem reads
\begin{equation}\label{Eq:MaxDistSo}\tag{$\bold P_{\max, \O,\Sigma}$}
\sup_{\substack{\Gamma\subset \partial \O\,, \mathcal H^{d-1}(\Gamma ) =V_0}}\int_{ \O} j\left(w^\Gamma\right).
\end{equation} 
The main result is the following Theorem:

\begin{theorem}\label{Th:BgbgDist}
Let $\O$ be a bounded open set of $\R^d$ such that $\partial\O$ is $\mathscr{C}^2$. 
Assume $f$ satisfies \eqref{Eq:Hypf} and $j$ satisfies \eqref{Eq:Hypj}.
Any solution $\beta^*$ of the optimisation problem \eqref{Eq:MaxDist} is bang-bang: there exists $\Gamma^*\subset \partial \O$ such that $ \beta^*=\mathds 1_{\Gamma^*}$. As a consequence, the shape optimisation problem \eqref{Eq:MaxDistSo} has a solution.
\end{theorem}Since the proof is very  similar to that of Theorem \ref{Th:BgbgBound}, we omit it in the main text of the article and only give it in Appendix \ref{Ap:Dist}.

\subsubsection{Non-existence and relaxation phenomenon for minimisation problems}
\begin{center}
\fbox{\textsf{Boundary criteria}}
\end{center}
Let us now consider the minimisation problem
\begin{equation}\tag{$\bold P_{\min,\partial \O,\mathcal B}$}\label{Eq:MinBound}
\min_{\beta\in \mathcal B(\partial \O)}\int_{\partial \O} j(u_\beta).
\end{equation} 
Once again, we refer to Lemma~\ref{Le:Exist} for the existence of optimal profiles. As stated hereafter, we shall show that the related shape optimisation problem
\begin{equation}\label{Eq:MinBoundSo}\tag{$\bold P_{\min,\partial \O,\Sigma}$}
\inf_{\substack{\Gamma\subset \partial \O\ \mathcal H^{d-1}(\Gamma )=V_0}}\int_{\partial \O} j\left(w^\Gamma\right)
\end{equation} exhibits a relaxation phenomenon. It is interesting to notice that the main argument for showing the second part of the following result rests upon a low frequency analysis of the second order optimality conditions.
\begin{theorem}\label{Th:RelaxBound}
Let $\O$ be a bounded open set of $\R^n$ such that $\partial\O$ is $\mathscr{C}^2$. 
Assume $f$ satisfies \eqref{Eq:Hypf} and $j$ satisfies \eqref{Eq:Hypj}.
\begin{itemize}
\item[(i)] Any solution $\beta^*$ of \eqref{Eq:MinBound} satisfies 
$$
\H\left(\{0<\beta^*<1\}\right)>0,
$$ 
so that \eqref{Eq:MinBoundSo} does not have a solution and enjoys a relaxation phenomenon.
\item[(ii)] Furthermore, let us introduce
$$
U_0(f):=\sup_{\beta\in \mathcal{B}(\partial\O)}\sup_{x\in \overline{\O}}u_\beta (x)\in (0,+\infty).
$$
There exists $C>0$ such that, if 
 $$
j''(u)\leq -Cj'(0)\quad \text{on}\quad [0,U_0(f)]\quad \text{and}\quad j'(U_0(f))>0,
$$then for any solution $\b$ of \eqref{Eq:MinBound} we have
$$\H\left(\{\b=1\}\right)>0.$$
\end{itemize}
\end{theorem}

Let us provide an example of function $j$ satisfying the assumptions of $(ii)$. Given $C>0$ and $U_0(f)>0$, the function $j$ given by
$$
j(u)=-\frac12 u^2+\frac12 \left(U_0(f)+\frac{1}{C}\right)u
$$
fulfills these conditions provided that $CU_0(f)< 1$. Since $U_0(f)$ does not depend on $C$, it suffices to chose $C$ small enough.

\begin{remark}
One has $U_0(f)<+\infty$ since $f\in L^\infty(\O)$. More precisely, one has
$$
\sup_{\beta\in \mathcal{B}(\partial\O)}\Vert u_\beta \Vert_{W^{1,p}(\O)}<+\infty
$$ 
for any $p$, by using standard elliptic regularity estimates which are detailed in Lemma \ref{Le:RegRobin}, and we conclude by using that the embedding $W^{1,p}(\O)\hookrightarrow \mathscr{C}^0(\overline{\O})$ is compact whenever $p$ is large enough.
\end{remark}
We refer to section \ref{Se:Tech} for comments on the proof.

%
%
%
%
\begin{center}
\fbox{\textsf{Distributed criteria}}
\end{center}

Here again, some of our methods naturally extend to the case of distributed criteria. The proof of the following result is very similar to that of Theorem \ref{Th:RelaxBound}, and we provide it in Appendix \ref{Ap:RelaxDist}. The minimisation problem under consideration is 
\begin{equation}\tag{$\bold P_{\min, \O,\mathcal B}$}\label{Eq:MinDist}
\min_{\beta\in \mathcal B(\partial \O)}\int_{ \O} j(u_\beta).
\end{equation} 

\begin{theorem}\label{Th:RelaxDist}
Let $\O$ be a bounded open set of $\R^n$ such that $\partial\O$ is $\mathscr{C}^2$. 
Assume $f$ satisfies \eqref{Eq:Hypf} and $j$ satisfies \eqref{Eq:Hypj}.
Then, any solution $\beta^*$ of \eqref{Eq:MinDist} satisfies 
$$
\Vol\left(\{0<\beta^*<1\}\right)>0.
$$ \end{theorem}

\subsection{Second case: energetic criteria}\label{Se:Results2}
Let us now tackle the two energetic optimisation problems

\begin{equation}\tag{$\bold{Q}_{\min}$}\label{Eq:MinNRJ}
\min_{\beta\in \mathcal B(\partial \O)}\int_{\O} fu_\beta\end{equation}
and
\begin{equation}\tag{$\bold{Q}_{\max}$}\label{Eq:MaxNRJ} \max_{\beta\in \mathcal B(\partial \O)}\int_{\O} fu_\beta.
\end{equation}
The existence of optimisers for \eqref{Eq:MinNRJ}-\eqref{Eq:MaxNRJ} can be obtained by adapting the arguments of Lemma \ref{Le:Exist} below.

As we have noted earlier, a salient feature of these problems is that they can be rewritten in terms of the energy of the equation \eqref{Eq:MainRobin}:\begin{equation}
\int_{\O} fu_\beta=-2\min_{u\in W^{1,2}(\O)}\frac12 \int_{\O}|\nabla u|^2+\frac12\int_{\O}\beta u^2-\int_{\O}fu.
\end{equation} 
Let us mention two important consequences of this fact: first, these problems are self-adjoint (\emph{i.e.} the adjoint state used to express the gradient of the criterion coincides with $u_\beta$); second, as an infimum of linear functionals is concave, the criterion is convex. This is why we can expect a more precise description of the optimisers of this problem. 

Let us mention that two very related contributions to the study of this problem are \cite{BUCUR2017451,Bucur2017}, in which several problems  of minimising some energetic criteria are studied. The main difference with our case is that the authors of \cite{BUCUR2017451,Bucur2017} rather study the problem of optimising such criteria with respect to $\beta$ for the boundary conditions $\beta\partial_\nu u_\beta+u_\beta=0$, which significantly changes the behaviour of the functionals.

We sum up our results in the following Theorem:

\begin{theorem}\label{Theo:minmaxNRJ}
Let $\O$ be a bounded open set of $\R^n$ such that $\partial\O$ is $\mathscr{C}^2$. 
Assume $f$ satisfies \eqref{Eq:Hypf} and $j$ satisfies \eqref{Eq:Hypj}.
\begin{itemize}
\item[(i)]  Every solution $\b$ of the maximization problem~\eqref{Eq:MaxNRJ} is bang-bang: there exists $\Gamma^*\subset \partial \O$ such that $ \beta^*=\mathds 1_{\Gamma^*}$.
\item[(ii)] Let $v_\O$ denote the solution of the Dirichlet problem
\begin{equation}\begin{cases}
-\Delta v_\O=f&\text{ in }\O,\\ 
v_\O=0&\text{ on }\partial \O,
\end{cases}\end{equation} 
and let $V_0^\O$ given by
$$
V_0^\O=-\frac{1}{\Vert \partial_\nu v_\O\Vert_{L^\infty(\partial \O)}}\int_{\partial\O}\partial_\nu v_\O\in (0,\mathcal H^{d-1}(\partial \O)].
$$
For every $V_0\in (0,V_0^\O)$, the minimization problem~\eqref{Eq:MinNRJ} enjoys a relaxation property and has a unique solution $\b$ given by
\begin{equation}\label{betaStarCasPart}
\b=V_0\frac{-\partial_\nu v_\O}{\int_{\partial\O}-\partial_\nu v_\O}.
\end{equation}
\item[(iii)] \textbf{Case where $\boldsymbol{f(\cdot)=1}$.} Let us assume that $f\equiv 1$. The constant admissible profile $\b_{V_0}\equiv\frac{V_0}{\H\left(\partial\O\right)}$ solves the minimization problem~\eqref{Eq:MinNRJ} if, and only if $\O$ is a ball.
\end{itemize}
\end{theorem}

\subsection{Non-linear models: distributed criteria}\label{Se:NL}
Since one of our initial motivation is also to understand some optimal control problems that arise in mathematical biology, where the state equation is typically non-linear, we now state the relevant results in this context. It should be noted that this case is, at the notational level, heavier than the linear one, but that the methods are similar to that of the linear case. For this reason, we only give the proofs in Appendix \ref{Ap:NL} and only state our theorem for distributed criteria.

\paragraph{Analytic set-up} We fix a non-linearity $g=g(x,y)$, and we first assume 
\begin{equation}\label{Eq:Hg}\tag{$\bold H_{NL}$}\text{ $g$ is measurable in both variables and $\mathscr C^2$ in the second variable.}\end{equation}

Keeping in mind the mathematical biology motivation, we are looking for non-negative solutions $y_\beta\in W^{1,2}(\O)$ of the equation 
\begin{equation}\label{Eq:MainNL}
\begin{cases}
-\Delta y_\beta=g(x,y_\beta)\,, &\text{ in }\O\,, 
\\\frac{\partial y_\beta}{\partial \nu}+\beta y_\beta=0&\text{ on }\partial \O\,, 
\\ y_\beta\geq 0\,, y_\beta\neq 0.\end{cases}\end{equation}
We of course assume:
\begin{equation}\label{Eq:HExist}\tag{$\bold H_{WP}$}\text{ For any $\beta \in \mathcal B(\partial \O)$, there exists a unique solution $y_\beta\in W^{1,2}(\O)$ to \eqref{Eq:MainNL}.}\end{equation}

Since we are working with optimality conditions, we need to be allowed to differentiate the map $\beta\mapsto y_\beta$. This is possible, granted the steady-states $y_\beta$ are linearly stable: in other words, letting $\mu_\beta$ be the first eigenvalue of the linearised operator, \emph{i.e.} 
\begin{equation}\label{Eq:DefMu}\mu_\beta:=\inf_{\p\in W^{1,2}(\O)\,, \int_\O \p^2=1}\left(\int_\O |\n \p|^2-\int_\O \frac{\partial g}{\partial y}(\cdot,y_\beta)\p^2+\int_{\partial \O}\beta \p^2\right)\end{equation} we must have 
\begin{equation}\label{Eq:HStab}\tag{$\bold H_{\operatorname{stab}}$}\forall \beta \in \mathcal B(\partial \O)\,, \mu_\beta>0.\end{equation}

Finally, $W^{1,p}$-estimates on $y_\beta$ are crucial. We hence need to ensure that 
\begin{equation}\label{Eq:HReg}\tag{$\bold H_{\operatorname{reg}}$}\forall \beta\in \mathcal B(\partial \O),\quad  \forall p\in [1;+\infty),\quad  y_\beta\in W^{1,p}(\O).\end{equation}

These are the only assumptions we need on $g$. After stating the theorem, we explain why a large class of monostable non-linearities satisfies these conditions.

\paragraph{Optimisation problem} We assume $g$ satisfies \eqref{Eq:Hg}-\eqref{Eq:HExist}-\eqref{Eq:HStab}-\eqref{Eq:HReg}.
We still work with a function $j$ satisfying \eqref{Eq:Hypj} and define 
$$\mathcal R_\O(\beta):=\int_\O j(y_\beta).$$
We consider the optimisation problems

\begin{equation}\label{Eq:PvNL}\tag{$\bold R_{\max}$}\max_{\beta\in \mathcal B(\partial \O)}
\mathcal R_\O(\beta)
\end{equation} as well as

\begin{equation}\label{Eq:PvNLMin}\tag{$\bold R_{\min}$}
\min_{\beta\in \mathcal B(\partial \O)}\mathcal R_\O(\beta).
\end{equation} 

 Under the assumption that $g$ satisfies  \eqref{Eq:Hg}-\eqref{Eq:HExist}-\eqref{Eq:HStab}-\eqref{Eq:HReg}, the existence of solutions to \eqref{Eq:PvNL}-\eqref{Eq:PvNLMin} is proved following the arguments of Lemma \ref{Le:Exist} below. Our main theorem is 
\begin{theorem}\label{Theo:NL}
Assume $j$ satisfies \eqref{Eq:Hypj} and $g$ satisfies  \eqref{Eq:Hg}-\eqref{Eq:HExist}-\eqref{Eq:HStab}-\eqref{Eq:HReg}. Then:
\begin{enumerate}
\item Any solution of \eqref{Eq:PvNL} is of bang-bang type. In other words, for every maximiser $\beta^*$ of \eqref{Eq:PvNL}, there exists a measurable subset $\Gamma^*$ of $\partial \O$ such that $\beta^*=\mathds 1_{\Gamma^*}$.
\item Any solution $\beta^*$ of \eqref{Eq:PvNLMin} satisfies 
$$
\H\left(\{0<\beta^*<1\}\right)>0.
$$ \end{enumerate}
\end{theorem}

\paragraph{An application to the optimal design of barriers in logistic models}
Let us give an example of a class of non-linearities $g$ satisfying \eqref{Eq:Hg}-\eqref{Eq:HExist}-\eqref{Eq:HStab}-\eqref{Eq:HReg}. We consider any function $m\in L^\infty(\O)$ such that 
\begin{equation}\label{Eq:IB} \int_\O m>\int_{\partial \O}\beta=V_0.\end{equation} Then the non-linearity 
$$g_\beta:(x,y)\mapsto y\left(m(x)-y\right)$$ satisfies \eqref{Eq:Hg}-\eqref{Eq:HExist}-\eqref{Eq:HStab}-\eqref{Eq:HReg}. To see why this is the case, we first observe that we are exactly in the context of monostable non-linearities, which, adapting the arguments of \cite{BHR}, yields the existence and uniqueness of a solution $y_\beta$ of \eqref{Eq:MainNL}, that further satisfies $y_\beta\leq \Vert m\Vert_{L^\infty}$. This readily gives the regularity estimates of Assumption \eqref{Eq:HReg}. Finally, we observe that, as $y_\beta\neq 0$ is a non-negative eigenfunction, associated with eigenvalue $0$, the monotonicity of the eigenvalue ensures that the first eigenvalue $\mu_\beta$ of $-\Delta -(m-2y_\beta)$ is positive for any $\beta$.

This equation models a population density that can access resources, accounted for in this scenario by the function $m$. For more references on the modelling on such phenomena we refer to \cite{MazariThese}, as well as to section \ref{Se:Bib} below. If we take $j(x)=x$, the optimisation problem \eqref{Eq:PvNL} (resp. \eqref{Eq:PvNLMin}) can be interpreted as: how should we design the features of the fence around the domain in order to maximise (resp. minimise) the population size? In this context, the bang-bang property has been deemed to be a relevant aspect of the qualitative analysis of optimisers \cite{MNP2021,NagaharaYanagida}. 

\subsection{Comments on the proof of the bang-bang property and technical context}\label{Se:Tech}
In this section, we comment upon our methods of proofs in order to provide some context regarding the tools we introduce to analyse the bang-bang property for the bilinear optimal control problems under consideration.
\paragraph{Oscillatory techniques used here}

 The bang-bang property is often proved \emph{via} the following reasoning: one determines the so-called "switch function" of the control problem. This function, say $\phi$, encodes the optimality conditions. If an optimiser is not bang-bang, this  usually implies that $\phi$ has a level-set of positive measure. To conclude, unique continuation theorems are invoked, that prove that this can not be the case since $\phi$ often satisfies a certain well-behaved optimality system.

However, this strategy is mostly useful for energetic problems. For non-energetic problems, such as the ones considered here, the equation solved by the switching function is usually not tractable. We refer for instance to \cite{MNP}. Instead, the method we introduce for Robin boundary conditions hinges on the one we introduced in \cite{MNP2021}. In it, we show that, for certain distributed bilinear optimal control problems, the second order derivative of the functional to optimise writes as something very similar to a Rayleigh quotient of a certain operator $\mathcal L$, in terms of $\dot u_\beta[h]$. Here, $\dot u_\beta[h]$ denotes the derivative of $\beta\mapsto u_\beta$ in a direction $h$. Concretely, this means that, the functionals under consideration in Theorems \ref{Th:BgbgBound}, \ref{Th:BgbgDist} and \ref{Theo:NL} being denoted generically $\mathcal J$ and double dots standing for second-order derivatives in a direction $h$, we derive an estimate of the form 
$$\ddot{\mathcal J}\gtrsim A\Vert \dot u_\beta[h]\Vert_{W^{1,2}(\O)}^2-B\Vert \dot u_\beta[h]\Vert_{L^2(\O)}^2$$ for some constants $A$ and $B$. The monotonicity of the functional enables to choose a positive $A$. We refer to Proposition \ref{Pr:Rayleigh} for a precise statement. 
Such an estimate allows to prove the bang-bang property: by assuming that a maximiser $\beta^*$ is not bang-bang, we show that there exists a perturbation $h$ that has a suitable support and such that $\dot u_\beta[h]$, has, in the spectral basis of $\mathcal L$, only high modes.  Thus, the second order derivative of the functional can be made arbitrarily high, and we can reach a contradiction. This idea is one of the key points of the proof of Theorems \ref{Th:BgbgBound}, \ref{Th:BgbgDist} and \ref{Theo:NL}. Of course, several points need to be refined in order to make this scheme suitable for the present context.

Second, and this is also a novelty of the article, we show that the same expression of the second order derivative in terms of Rayleigh quotients allows to prove relaxation phenomena for minimisation problem. Contrary to the bang-bang property, the method rests upon the use of low-eigenmodes, and it is at the center of Theorem \ref{Th:RelaxBound} and \ref{Th:RelaxDist}.

\paragraph{Relationship with other existence theorems in shape optimisation}
A particularly crucial step in all our proofs is the monotonicity of the functionals we optimise. Concretely, consider the problem \eqref{Eq:MaxBound}. Then one of the central points of the proof of Theorem \ref{Th:BgbgBound} is that the functional $\mathcal J_{\partial \O}(\beta)=\int_{\partial \O} j(u_\beta)$ is increasing. This implies that, from a shape optimization perspective, the functional $\Gamma\mapsto \int_{\partial \O} j(w^{\Gamma})$ is increasing. In this context, it is tempting to use the seminal theorem of Buttazzo-DalMaso \cite{BDM} to conclude that there exists a solution of \eqref{Eq:MaxBoundSo}. However, the topology on sets required by \cite{BDM} is not suited to our type of problems: in \cite{BDM}, the convergence on sets is the $\Gamma$-convergence; unfortunately, it is not clear that this topology makes the functional we consider here continuous, which prohibits using this result. Furthermore, the strategies of proof are very different. Nevertheless, we wish to highlight the fact that, as in \cite{BDM}, the crux of the problem is the monotonicity of the functional.

\subsection{Related works and bibliographical references}\label{Se:Bib}

\paragraph{Optimisation of criteria involving the Robin coefficients} Two lines of research coexist when it comes to optimisation in combination with Robin boundary conditions, one dedicated to the type of problems we consider here, where the domain $\O$ is fixed and $\beta$ is the variable, and another one focused on optimising the domain $\O$ itself. Some contributions combine these two approaches.

Among the vast literature relevant in such queries, let us first single out \cite{MR3009728}  and \cite{MR1623313}. Both papers deal with time-dependent equation, and focus on "tracking-type" functionals. In other words, in the framework of our paper, this would mean that the functional to optimise would involve a term of the form $\Vert u_\beta-u_{\operatorname{ref}}\Vert_X^2$ for some norm $X$ and some reference $u_{\operatorname{ref}}$. Such criteria are known to be very relevant in many applications. Our paper, on the other hand, focuses on "free" functionals, and the methods used to analyse each problem are necessarily very different.

The aforementioned \cite{BUCUR2017451,Bucur2017} investigate the  properties of the optimisers $\beta$ that minimise the natural energy of the model or some related eigenvalues. Although Theorem \ref{Theo:minmaxNRJ} is closely related to these contributions, let us highlight the fact that they consider boundary conditions of the form $\beta \partial_\nu u+u=0$, which changes the features of the problems. This contribution falls into the first category described above (energetic functionals), but it is notable that the motivation in \cite{BUCUR2017451,Bucur2017} is an optimal insulation problem which was then deeply explored from the point of view of both optimal control and shape optimisation in the recent {\cite{DellaPietra2021}}. Notable in the latter is the fact that one of their main results, \cite[Theorem 5.1]{DellaPietra2021} uses Talenti-like inequalities for Robin boundary conditions, which has been another very active line of research following \cite{Alvino2019ATC}.

\paragraph{Bilinear control problems}

Let us first underline that the study of bilinear controllability of systems (\emph{i.e.} trying to reach an exact state using a bilinear control) is a very active field. We refer, for instance, to \cite{Alabau,Beauchard2010,Floridia}. 

On the other hand, bilinear optimal control problems, in which one rather aims at optimising a certain criterion, as is the case in the present paper, have received less attention. We have already mentioned several contributions related to the optimisation of Robin coefficients in order to minimise tracking-type functionals; in this broader context of bilinear optimisation, let us also refer to \cite{Fister,GuillnGonzlez2020}, where bilinear optimal control problems for chemotaxis or chemorepulsion models are studied.  The functional the authors wish to minimise is also of tracking type, but the control acts on the interior of the domain rather than on the boundary.  Most of the emphasis is placed on deriving existence results and optimality conditions. Related to these contributions, we point, for another distributed bilinear optimal control problem, to \cite{Borzi},  a contribution that focuses on a numerical multigrid analysis of the optimisation system.  

However, the literature is scarcer when it comes to the qualitative analysis of bilinear optimal control problems when the functionals is not of tracking type. A paradigmatic example is that of the optimisation of the total population size in monostable models. In this problem, which originated in \cite{LouInfluence} the goal is to spread resources so as to maximise the integral of the solution of a reaction-diffusion equations. In the elliptic case, proving the bang-bang property for optimal resources distributions proves surprisingly difficult, and the problem exhibits a very intricate qualitative behaviour. We refer to \cite{BaiHeLi,BINTZ2020,LouNagaharaYanagida,MNP,MRBSIAP,NagaharaYanagida} and the references therein for an overview of this problem. It should be noted that the bang-bang property was only recently proved by the two authors in collaboration with a third one in \cite{MNP2021}, using a method that, as was explained, is linked to the one we develop here for Robin controls.

\section{Preliminary material}\label{Se:Prelim}

We first give some existence results and regularity estimates for the solutions of \eqref{Eq:MainRobin}. Since the proofs are standard, we only give them below in Appendix \ref{Ap:Prelim}.
We begin with a uniform estimate:
\begin{lemma}\label{Le:Est}Let $\O$ be a bounded domain in $\R^d$ with a $\mathscr C^2$ boundary. There exists $C>0$ such that
$$
\forall \beta\in \mathcal B(\partial \O),\quad \forall v\in W^{1,2}(\O),\quad  C\Vert v\Vert_{W^{1,2}(\O)}^2\leq \int_\O |\nabla v|^2 +\int_{\partial\O}\beta v^2.
$$
\end{lemma}\begin{lemma}\label{Le:RegRobin}Assume $f\in L^\infty(\O)$. Then the equation \eqref{Eq:MainRobin} has a unique solution $u_\beta\in W^{1,2}(\O)$. Furthermore, for any $p\in [1;+\infty)$, $u_\beta\in W^{1,p}(\O)$, 
$$\sup_{\beta\in \mathcal B(\partial \O)}\Vert u_\beta\Vert_{W^{1,p}(\O)}<\infty$$ and
$$\forall \beta \in \mathcal B(\partial \O)\,, \inf_{\overline \O} u_\beta>0.$$ 

In particular, as a consequence of Sobolev embeddings,  we have the uniform estimates
$$0<\inf_{\beta\in \mathcal B(\partial \O)}\inf_\O u_\beta\leq \sup_{\beta\in \mathcal B(\partial \O)}\sup_\O u_\beta<\infty.$$

\end{lemma}
This lemma is a standard consequence of Robin regularity results. 

We finally state our existence result.
\begin{lemma}\label{Le:Exist}
Let us assume that $\partial\O$ is $\mathscr{C}^2$. Each of the problems \eqref{Eq:MaxBound}-\eqref{Eq:MinBound}-\eqref{Eq:MaxDist}-\eqref{Eq:MinDist} has a solution.
\end{lemma}
\section{Proofs of Theorems \ref{Th:BgbgBound} and \ref{Th:RelaxBound}}
Throughout this section, $f$ and $j$ are assumed to satisfy \eqref{Eq:Hypf} and \eqref{Eq:Hypj} respectively. We work only on boundary criteria;  thus, to alleviate notation, we drop the subscript in $\mathcal J_{\partial \O}$ and simply write 
$$\mathcal J(\beta):=\int_{\partial \O} j(u_\beta).$$ We begin by computing the derivatives of $\mathcal J$, as these derivatives are key in proving Theorems \ref{Th:BgbgBound} and \ref{Th:RelaxBound}.

\subsection{Computation of the derivatives of $\mathcal J_{\partial \O}$}\label{Se:OptBound}
The differentiability of the map $\mathcal B(\partial \O)\ni \beta\mapsto u_\beta\in W^{1,2}(\O)$ is standard. Let us consider $\beta\in \mathcal B(\partial \O)$ and an admissible perturbation\footnote{\label{footnotehAdm}The wording ``admissible perturbation'' means that $h$ belongs to the tangent cone to the set $\mathcal B(\partial \O)$ at $\beta$.
It corresponds to the set of functions $h\in L^\infty(\O)$ such that, for any sequence of positive real numbers $\varepsilon_n$ decreasing to $0$, there exists a sequence of functions $h_n\in L^\infty(\O)$ converging in $L^2(\O)$ to $h$ as $n\rightarrow +\infty$, and $\beta+\varepsilon_nh_n\in\mathcal B(\partial \O)$ for every $n\in\N$.} $h$ at $\beta$ 
We denote with a single (resp. double) dot the first (resp. second) order Gateaux derivative of relevant quantities at $\beta$ in the direction $h$. By differentiating \eqref{Eq:MainRobin}, we see that $\dot u_\beta$ solves
\begin{equation}\label{Eq:DotuBound}\begin{cases}
-\Delta \dot u_\beta=0&\text{ in }\O, 
\\ \frac{\partial \dot u_\beta}{\partial \nu}+\beta \dot u_\beta=-hu_\beta &\text{ on }\partial \O,
\end{cases}
\end{equation}
while $\ddot u_\beta$ satisfies
\begin{equation}\label{Eq:DdotuBound}\begin{cases}
-\Delta \ddot u_\beta=0&\text{ in }\O, 
\\ \frac{\partial \ddot u_\beta}{\partial \nu}+\beta \ddot u_\beta=-2h\dot u_\beta &\text{ on }\partial \O.
\end{cases}
\end{equation}
Existence and uniqueness of $W^{1,2}(\O)$ solutions to these equations are immediate.
The derivatives of the criterion $\mathcal J$ are, similarly, given by 
\begin{equation}\label{Eq:SV}\dot{\mathcal J}(\beta)[h]=\int_\O \dot u_\beta j'(u_\beta)\text{ and } \ddot{\mathcal J}(\beta)[h,h]=\int_{\partial \O}\ddot u_\beta j'(u_\beta)+\int_{\partial \O}\left(\dot u_\beta\right)^2 j''(u_\beta).
\end{equation}
In order to make these derivatives more tractable, we introduce the adjoint state $p_\beta$ as the unique  solution in $W^{1,2}(\O)$ of 
\begin{equation}\label{Eq:Adjoint}
\begin{cases}
-\Delta p_\beta=0&\text{ in }\O, 
\\ \frac{\partial p_\beta}{\partial \nu}+\beta p_\beta=j'(u_\beta)&\text{ on }\partial \O.\end{cases}
\end{equation}
Since $\inf_\O u_\beta>0$ and since $j'>0$ on $\R_+^*$, it follows from the maximum principle that $p_\beta$ is positive in $\overline \O$ and that we even have
\begin{equation}\label{Eq:CG}\inf_{\overline \O}p_\beta>0.\end{equation}
Multiplying \eqref{Eq:Adjoint} by $\dot u_\beta$ and \eqref{Eq:DotuBound} by $p_\beta$, an integration by parts yields 
\begin{align*}
0&=\int_\O \n \dot u_\beta\cdot \n p_\beta+\int_{\partial \O}\beta p_\beta \dot u_\beta-\int_{\partial \O}j'(u_\beta)\dot u_\beta
\\0&=\int_\O \n \dot u_\beta\cdot \n p_\beta+\int_{\partial \O}\beta p_\beta \dot u_\beta+\int_{\partial \O} hu_\beta p_\beta 
\end{align*}
Hence, we obtain 
$$
\int_{\partial \O} j'(u_\beta)\dot u_\beta=-\int_{\partial \O}hu_\beta p_\beta,
$$
by combining the two identities above.
In particular, according to \eqref{Eq:SV}, we get 
\begin{equation}\label{Eq:OptO1}
\dot{\mathcal J}(\beta)[h]=-\int_{\partial \O}h u_\beta p_\beta.
\end{equation}

 Multiplying \eqref{Eq:Adjoint} by $\ddot u_\beta$, and \eqref{Eq:DdotuBound} by $p_\beta$, and integrating by parts also gives the following expression for the second-order Gateaux derivative of $\mathcal J$:
 \begin{equation}\label{Eq:Opt2}\ddot{\mathcal J}(\beta)[h,h]=-2\int_{\partial \O}h\dot u_\beta p_\beta+\int_{\partial \O}j''(u_\beta)\left(\dot u_\beta\right)^2.
\end{equation}

\subsection{Proof of Theorem \ref{Th:BgbgBound}}\label{sec:proofTh:BgbgBound}
We shall argue by contradiction: let us fix a maximiser $\beta^*\in \mathcal B(\partial \O)$ (which exists thanks to Lemma \ref{Le:Exist}) such that the set \begin{equation}\label{Eq:Anormal}\omega^*:=\{0<\beta^*<1\}\end{equation} has positive measure:
$$\H(\omega^*)>0.$$ It follows that for any admissible perturbation $h$ at $\b$ supported in $\omega^*$ (in the sense of Footnote~\ref{footnotehAdm}), there holds 
\begin{equation}\dot{\mathcal J}(\b)[h]=0.\end{equation} To reach a contradiction, it suffices to prove that there exists an admissible perturbation $h$ supported in $\omega^*$ such that 
\begin{equation}\ddot{\mathcal J}(\b)[h,h]>0.\end{equation} We start by recalling that from \eqref{Eq:Opt2}, we have
 \begin{equation}\ddot{\mathcal J}(\b)[h,h]=-2\int_{\partial \O}h\dot u_\b p_\b+\int_{\partial \O}j''(u_\b)\left(\dot u_\b\right)^2.
\end{equation} 

The first step of this proof is to obtain an expression of $\ddot{\mathcal J}$ that is reminiscent of a Rayleigh quotient. This is the purpose of the following proposition:
\begin{proposition}\label{Pr:Rayleigh} There exist three constants $A\,, B\,, C$ with $A>0$ such that 
\begin{equation}\label{Eq:Vivaldi}
\ddot{\mathcal J} (\b)[h,h]\geq A\int_\O  |\n \dot u_\b|^2-B\Vert  \dot u_\b\Vert_{W^{1,2}(\O)}\Vert \dot u_\b\Vert_{L^2(\O)}  -C\int_{\partial \O}\dot u_\b^2.
\end{equation}
\end{proposition}

\begin{proof}[Proof of Proposition \ref{Pr:Rayleigh}]
 We isolate the first part of \eqref{Eq:Opt2}, and define 
\begin{equation}{\mathcal W}(\b)[h,h]:=-2\int_{\partial \O} h\dot u_\b p_\b.\end{equation} Using the boundary condition of the equation \eqref{Eq:DotuBound} on $\dot u_\b$ this quantity rewrites
$${\mathcal W}(\b)[h,h]=2\int_{\partial \O}\left(\partial_\nu \dot u_\b+\b \dot u_\b\right)\frac{ p_\b}{u_\b}\dot u_\b.$$ Let us introduce the function $\Psi_\b$ given by
$$\Psi_\b:=\frac{p_\b}{u_\b},$$  so that
\begin{equation}\mathcal W(\b)[h,h]=\int_{\partial \O}\Psi_\b  \left(\partial_\nu\left({\dot u_\b}^2\right)+2\b \dot u_\b^2\right).\end{equation} Since $\inf_{\overline \O}u_\b, p_\b>0$ from \eqref{Eq:CG} and lemma \ref{Le:RegRobin}, and since  $\sup_{\overline \O}u_\b, p_\b<\infty$, one has
$$
\inf_{\overline \O}\Psi_\b>0.
$$
Furthermore, for every function $z\in W^{2,2}(\O)$, one has
\begin{equation}\label{Eq:B}-\int_\O \left(\Delta \Psi_\b\right) z+\int_\O \left(\Delta z\right)\Psi_\b=\int_{\partial\O}\left( \partial_\nu z\right)\Psi_\b -\int_{\partial\O}z\left(\partial_\nu \Psi_\b\right) .\end{equation}
We first compute
\begin{equation}\frac{\partial \Psi_\b}{\partial \nu}=\frac{\partial_\nu p_\b}{u_\b}-\Psi_\b \frac{\partial_\nu u_\b}{u_\b}=-\b \Psi_\b+\b\Psi_\b+\frac{j'(u_\b)}{u_\b}=\frac{j'(u_\b)}{u_\b}>0.\end{equation} 
We want to take $z=\dot u_\b^2$ in \eqref{Eq:B}. We have 
$$
\Delta z=2|\n \dot u_\b|^2+2\dot u_\b \Delta\dot u_\b =2|\n \dot u_\b|^2\text{ since $\Delta \dot u_\b=0$}.
$$ 
Hence, 
\begin{equation}
\int_{\partial \O}\Psi_\b \partial_\nu \left({\dot u_\b^2}\right)=2\int_\O |\n \dot u_\b|^2\Psi_\b-\int_\O {\dot u_\b^2}{}\Delta \Psi_\b+\int_{\partial \O}\frac{j'(u_\b)}{u_\b}\dot u_\b^2
\end{equation}
Thus, the second order derivative of ${\mathcal J}$ rewrites
\begin{equation}\ddot{\mathcal J} (\b)[h,h]=2\int_\O \Psi_\b |\n \dot u_\b|^2-\int_\O {\dot u_\b^2}\Delta \Psi_\b+\int_{\partial \O}\left(2\b\Psi_\b+\frac{j'(u_\b)}{u_\b}+j''(u)\right) \dot u_\b^2.
\end{equation}
Since $\b$ and $u_\b$ belong to $L^\infty(\partial \O)$ and since $j\in \mathscr C^2$, there exists a constant $C$ independent of $h$ such that, for any admissible perturbation $h$ 
\begin{equation}
\int_{\partial \O}\left(2\b\Psi_\b+\frac{j'(u_\b)}{u_\b}+j''(u)\right)\left(\dot u_\b\right)^2\geq -C \int_{\partial \O } \dot u_\b^2.
\end{equation}
Overall, we thus have the following estimate on $\ddot{\mathcal J}$:
\begin{equation}\label{Eq:Ete90}
\ddot{\mathcal J}(\b)[h,h]\geq 2\int_\O \Psi_\b |\n \dot u_\b|^2-\int_\O {\dot u_\b^2}\Delta \Psi_\b-C\int_{\partial \O } \dot u_\b^2.
\end{equation}
Since $\inf_\O \Psi_\b>0$, there exists $A>0$ such that
\begin{equation}\ddot{\mathcal J} (\b)[h,h]\geq A\int_\O  |\n \dot u_\b|^2-\int_\O {\dot u_\b^2}\Delta \Psi_\b-C\int_{\partial \O}\dot u_\b^2.
\end{equation}
Direct computations yield
\begin{eqnarray*}
\Delta \Psi_\b&=&\frac{\Delta p_\b}{u_\b}-\Psi_\b \frac{\Delta u_\b}{u_\b}-2\frac{\n u_\b\cdot \n p_\b}{u_\b^2}+2\Psi_\b \frac{|\n u_\b|^2}{u_\b^2}\\
&=& \frac{f\Psi_\b}{u_\b} -2\frac{\n u_\b\cdot  \n p_\b}{u_\b^2}+2\Psi_\b \frac{|\n u_\b|^2}{u_\b^2}.
\end{eqnarray*}
Besides, since $u_\b$ and $p_\b$ belong to $W^{1,p}(\O)$ for any $p\in[ 1;+\infty)$ and since $\inf_{\overline \O}u_\b>0$, it follows that, for any $p\geq 1$,
$$
-2\frac{\n u_\b\cdot \n p_\b}{u_\b^2}+2\Psi_\b \frac{|\n u_\b|^2}{u_\b^2}\in L^p(\O).$$ 
Therefore, $\Delta \Psi_\b$ belongs to $L^p(\O)$ for any $p\geq 1$. Now, let us apply  H\"{o}lder's inequality with three exponents: let $(q,r)\in (1;+\infty)^2$ be such that 
$$
\frac12+\frac1q+\frac1r=1\text{ and } r\in (2;2^*)\quad \text{with }2^*:=\frac{2n}{n-2}.$$ 
By H\"{o}lder's inequality, we obtain 
\begin{equation}\left|\int_\O {\dot u_\b^2}\Delta \Psi_\b\right|\leq \Vert \Delta \Psi_\b\Vert_{L^q(\O)}\Vert \dot u_\b\Vert_{L^r(\O)}\Vert \dot u_\b\Vert_{L^2(\O)}\end{equation} and so, from the Sobolev embedding $W^{1,2}(\O)\hookrightarrow L^r(\O)$, there exists a constant $B>0$ such that

\begin{equation}\left|\int_\O {\dot u_\b^2}\Delta \Psi_\b\right|\leq B\Vert \dot u_\b\Vert_{W^{1,2}(\O)}\Vert \dot u_\b\Vert_{L^2(\O)}.\end{equation}
We thus infer that
\begin{equation}\ddot{\mathcal J} (\b)[h,h]\geq A\int_\O  |\n \dot u_\b|^2-B\Vert  \dot u_\b\Vert_{W^{1,2}(\O)}\Vert \dot u_\b\Vert_{L^2(\O)}  -C\int_{\partial \O}\dot u_\b^2.
\end{equation}
As a consequence, \eqref{Eq:Vivaldi} is proved. 

This concludes the proof of the Proposition.
\end{proof}

We point that since the standard $W^{1,2}(\O)$ norm is equivalent to $W^{1,2}(\O)\ni u\mapsto \left(\int_\O |\n u|^2+\int_{\partial \O} u^2\right)^{1/2}$ (we refer to Lemma \ref{Le:Est} for instance), in what follows, we use 
$$\Vert u\Vert_{W^{1,2}(\O)}^2=\int_\O |\n u|^2+\int_{\partial \O} u^2.$$ Up to multiplying $B$ by a positive constant, \eqref{Eq:Vivaldi} remains unchanged and we thus keep the notation $B$. We turn back to the proof of the Theorem.

The key point is to construct an admissible perturbation $h$, supported in $\omega^*$ (introduced in \eqref{Eq:Anormal}) such that $\ddot{\mathcal J}(\b)[h,h]>0$.
To this aim, we use \eqref{Eq:Vivaldi} to say that it suffices to construct, for $\e>0$ and $\delta>0$ small enough, an admissible perturbation $h$ supported in $\omega^*$ such that $\Vert h\Vert_{L^2(\partial\O)}=1$ and that satisfies
\begin{equation}\label{Eq:C1}
\Vert \dot u_\b\Vert_{L^2(\O)} \leq \e \Vert \dot u_\b\Vert_{W^{1,2}(\O)}
\end{equation}
and
\begin{equation}\label{Eq:C2}
\Vert \dot u_\b\Vert_{L^2(\partial \O)} \leq \sqrt\delta \Vert \n\dot u_\b\Vert_{L^2(\O)}
\end{equation}
Indeed, if $h$ is non-zero and satisfies \eqref{Eq:C1}-\eqref{Eq:C2} then, according to \eqref{Eq:Vivaldi} we obtain
\begin{align*}
\ddot{\mathcal J}(\b)[h,h]&\geq A\Vert \n\dot u_\b\Vert_{L^2(\O)}^2-B\Vert \n\dot u_\b\Vert_{W^{1,2}(\O)}\Vert \dot u_\b\Vert_{L^2(\O)}-C\Vert \dot u_\b\Vert_{L^2(\partial\O)}^2
\\&\geq (A-C\delta)\Vert \n\dot u_\b\Vert_{L^2(\O)}^2-B\e\Vert \dot u_\b\Vert_{W^{1,2}(\O)}^2
\\&\geq(A-C\delta)\Vert \n\dot u_\b\Vert_{L^2(\O)}^2- B\e\left(\Vert\n \dot u_\b\Vert_{L^{2}(\O)}^2+\Vert \dot u_\b\Vert_{L^{2}(\partial\O)}^2\right)
\\&\geq(A-C\delta- B\e(1+\delta))\Vert \n\dot u_\b\Vert_{L^2(\O)}^2.
\end{align*}
In particular, to ensure that we have a positive right-hand side it suffices to pick $\e$ and $\delta$ small enough.

It thus remains to prove that such a perturbation $h$ exists. Let us highlight that we will obtain \eqref{Eq:C1} and \eqref{Eq:C2} by two different paths. We start with \eqref{Eq:C2}.

\paragraph{Regarding condition \eqref{Eq:C2}.} 
Let us fix an arbitrary $\delta>0$. To prove that we can choose $h$ supported in $\omega^*$ such that $\Vert h\Vert_{L^2(\partial\O)}=1$ and that satisfies \eqref{Eq:C2} we rely on eigenvalues and eigenfunctions of a Robin-Steklov type operator. More precisely, we introduce the Hilbert basis $\{\phi_k\}_{k\in \N}$ of $L^2(\partial\O)$ given by 
\begin{equation}\label{def:phik}
\begin{cases}
-\Delta \phi_k=0&\text{ in }\O, 
\\\frac{\partial \phi_k}{\partial \nu}+\beta_k\phi_k=\sigma_k\phi_k&\text{ on }\partial \O, 
\end{cases}\quad\text{ where }0<\sigma_0\leq \sigma_1\leq \dots\leq\sigma_k\underset{k\to \infty}\to \infty  \text{ and }\quad \int_{\partial \O}\phi_k\phi_{k'}=\delta_{k,k'}.
\end{equation}
 We prove in Appendix~\ref{sec:appendSteklov} how such eigenelements are defined; our arguments follow the classical \cite[Section 11]{MR2072072}. 

To see how this elements enable us to obtain \eqref{Eq:C2}, let us first observe that for any admissible $h$, the function $-hu_\b$ belongs to $L^2(\partial \O)$, and thus expands as
\begin{equation}-hu_\b=\sum_{k=0}^\infty \alpha_k(h)\phi_k\end{equation} where 
\begin{equation}
\forall k\in \N\,, 
\alpha_k(h)=\int_{\partial \O}(-hu_\b)\phi_k
\end{equation} 
Then, since $\dot u_\b$ solves
$$
\begin{cases}-\Delta u_\b=0&\text{ in }\O, \\
\frac{\partial u_\b}{\partial \nu}+\b u_\b=\sum_{k=0}^\infty \alpha_k(h)\phi_k&\text{ on }\partial\O,
\end{cases}
$$ 
we have 
\begin{equation} \dot u_\b=\sum_{k=0}^\infty\frac{\alpha_k(h)}{\sigma_k}\phi_k,\end{equation} which then allows us to compute
\begin{equation}
\Vert \dot u_\b\Vert_{L^2(\partial \O)}^2=\sum_{k=0}^\infty\frac{\alpha_k(h)^2}{\sigma_k^2}\quad \text{and} \quad 
\Vert \n \dot u_\beta\Vert_{L^2(\O)}^2+\int_{\partial \O}\b \dot u_\b^2=\sum_{k=0}^\infty \frac{\alpha_k(h)^2}{\sigma_k}.
\end{equation}
Therefore, since $0\leq \beta^*\leq 1$, one has
$$
\Vert \n \dot u_\beta\Vert_{L^2(\O)}^2\geq \sum_{k=0}^\infty \frac{\alpha_k(h)^2}{\sigma_k}-\int_{\partial \O} \dot u_\b^2\geq \sum_{k=0}^\infty \alpha_k(h)^2\left(\frac1{\sigma_k}-\frac1{\sigma_k^2}\right).
$$
Observe that if $K\in \N^*$ is chosen in such a way that
\begin{equation}\label{Eq:FL}
\forall k\in \llbracket 0, K-1\rrbracket,\quad \alpha_k(h)=0,
\end{equation} 
the previous estimates imply
\begin{equation}\Vert \n \dot u_\beta\Vert_{L^2(\O)}^2\geq \sum_{k=K}^\infty \alpha_k(h)^2\left(\frac1{\sigma_k}-\frac1{\sigma_k^2}\right)=\sum_{k=K}^\infty\frac{\alpha_k(h)^2}{\sigma_k^2}(\sigma_k-1)\geq \left(\sigma_K-1\right)\Vert \dot u_\b\Vert_{L^2(\partial\O)}^2.\end{equation}
Thus if we fix $K\in \N^*$ such that 
\begin{equation}\label{Eq:Patin} \frac{1}{{\sigma_K-1}}\leq {\delta}\end{equation} and if we pick $h$ such that \eqref{Eq:FL} holds, we reach condition \eqref{Eq:C2}. We now prove that for any $K\in \N$, there exists an admissible perturbation $h\neq 0$ supported in $\omega^*$ such that \eqref{Eq:FL} is satisfied.

Let $K\in \N^*$ such that \eqref{Eq:Patin} is satisfied be fixed. According to the discussion above, we want to prove that there exists $h\in L^2(\partial \O)$ supported in $\omega^*$ such that 
\begin{enumerate}
\item$ \Vert h\Vert_{L^2(\partial \O)}^2=1,$
\item $\int_{\partial \O}h=0$,
\item $\forall k\in\llbracket 0,K-1\rrbracket$, $\alpha_k(h)=0$.
\end{enumerate}
 $\H(\omega^*)>0$ so $L^2(\omega^*)$ is infinite dimensional. We introduce the following family of $(K+1)$ linear forms on $L^2(\omega^*)$: \begin{eqnarray*}
 R:L^2(\omega^*)\ni h&\mapsto& \int_{\omega^*}h\\
T_k:L^2(\omega^*)\ni h&\mapsto &\int_{\omega^*} hu_\b \phi_k
\end{eqnarray*}
for all $k\in \llbracket 0,K-1\rrbracket$. Since $u_\b$ belongs to $L^\infty(\partial\O)$, each $T_k$ defines a continuous linear form and $R$ is itself obviously continuous. As a consequence, the subspace
 \begin{equation}\label{Eq:DefE}
 E_\delta:=\ker(R)\bigcap\left(\underset{k=0}{\overset{K}\bigcap}\ker(T_k)\right)
 \end{equation}
 has finite co-dimension. Hence, there exists, in particular, $h\in E_\delta$ such that $\Vert h\Vert_{L^2(\omega^*)}=1$. It suffices to extend $h$ by 0 to $\O$  to obtain an admissible perturbation that satisfies all the required conditions.

\paragraph{Satisfying both conditions \eqref{Eq:C1} and \eqref{Eq:C2}.} 
Thus, for a fixed $\delta>0$ (not necessarily small), every $h\in E_\delta$ satisfies \eqref{Eq:C2}, where $E_\delta$ is defined in \eqref{Eq:DefE}. We have also fixed $\e>0$. Let us show that there exists $h\in E_\delta$ that satisfies \eqref{Eq:C1}, which suffices to conclude the proof. 
In other words, we will  prove that
\begin{equation}\label{Eq:G} 
\forall C>0, \quad \exists h \in E_\delta, \quad \Vert \dot u_{\b}\Vert_{W^{1,2}(\O)}> C\Vert \dot u_{\b}\Vert_{L^2(\O)}.
\end{equation} 

To prove \eqref{Eq:G}, let us argue by contradiction, assuming the existence of $C>0$ such that
\begin{equation}\label{Eq:G1}
\forall h \in E_\delta, \quad \Vert \dot u_{\b}\Vert_{W^{1,2}(\O)}\leq C\Vert \dot u_{\b}\Vert_{L^2(\O)}.
\end{equation}
In what follows, we will rather denote $\dot u_\b$ by $\dot u_\b[h]$ to emphasize the dependency of this function in $h$.
Let us introduce $$X_\delta:=\left\{\dot u_{\b}[h], h\in E_\delta \right\}.$$ $X_\delta$ is a subspace of $W^{1,2}(\O)$ and since the map $E_\delta\ni h\mapsto \dot u_{\b}[h]$ is an injection, $X_\delta$ is infinite dimensional. Consequently, there exists an $L^2(\O)$-orthonormal family $\{v_k\}_{k\in \N}$ in $X_\delta$. In particular, for any $k\in \N$, $\Vert v_k\Vert_{L^2(\O)}=1$. Furthermore, according to the Parseval inequality, one has 
$$v_k\underset{k\to \infty}\rightharpoonup 0\text{ weakly in $L^2(\O)$}.$$ However, should \eqref{Eq:G1} hold, the family $\{v_k\}_{k\in \N}$ would be uniformly bounded in $W^{1,2}(\O)$ and thus, by the Rellich-Kondrachov theorem, converge strongly in $L^2$ to a closure point $v_\infty$ (up to a subsequence). Since we already know it converges weakly to 0, one must have $v_\infty=0$ on the one hand and $\Vert v_\infty\Vert_{L^2(\O)}=1$ on the other hand, leading to a contradiction. The conclusion follows: there necessarily exists $h$ such that \eqref{Eq:C1}-\eqref{Eq:C2} holds. 

The proof of the Theorem is now complete.

\subsection{Proof of Theorem \ref{Th:RelaxBound}}
We prove each point of Theorem \ref{Th:RelaxBound} separately. Once again, notational convenience leads us tu dropping the $\partial \O$ subscript, and to just writing
$$\mathcal J(\beta)=\int_{\partial \O}j(u_\beta).$$
\paragraph{Proof of $\boldsymbol{(i)}$.}
Before we get to the core of the Theorem, let us point out a consequence of the expression of the first order derivative of the criterion given in \eqref{Eq:OptO1}. We set, for any $\beta\in \mathcal B(\partial \O)$, 
$$\Phi_\beta:=u_\beta p_\beta$$ where $p_\beta$ is the solution of \eqref{Eq:Adjoint}. Therefore, for any admissible perturbation $h$ at a given $\beta$, there holds
\begin{equation}\dot{\mathcal J}(\beta)[h]=-\int_{\partial \O} h \Phi_\beta.\end{equation} 
Let $\beta^*$ be  a solution of \eqref{Eq:MinBound}. Since we must have 
$$-\dot{\mathcal J}(\beta^*)[h]=\int_{\partial \O}\Phi_{\beta^*}h\leq 0$$ for any admissible perturbation $h$ at $\beta^*$, there exists a real number $\lambda$ (necessarily positive as $\Phi_\beta>0$) such that 
\begin{enumerate}
\item $\{0<\beta^*<1\}\subset\left\{\Phi_{\beta^*}=\lambda\right\}$,
\item $\{\beta^*=1\}\subset \{\Phi_{\beta^*}\geq \lambda\}$,
\item $\{\beta^*=0\}\subset \{\Phi_{\beta^*}\leq \lambda\}$.
\end{enumerate}
We shall now prove that for any solution $\beta^*$ of \eqref{Eq:MinBound}, such optimality conditions imply that 
\begin{equation}\label{Eq:Yellow}
\H\left(\{\beta^*=0\}\right)=0,
\end{equation} 
which necessarily yields  
\begin{equation}\label{Eq:Yellow2}
\H\left(\{0<\beta^*<1\}\right)>0,
\end{equation} the required conclusion. Indeed, this follows from the volume constraint $\int_{\partial\O}\b=V_0$.

To prove \eqref{Eq:Yellow}, we first compute the equation satisfied by the function $\Phi_{\beta^*}=u_\b p_\b$. By direct computation, we have
\begin{equation}\label{Eq:10}
\n\Phi_{\beta^*}=u_\b\n p_\b+p_\b\n u_\b, -\Delta \Phi_\b=-p_\b\Delta u_\b-u_\b \Delta p_\b-2 \n u_\b\cdot \n p_\b.
\end{equation}
Let us first set $$B:=-2\frac{\n p_\b}{p_\b}$$ and $$V:=\frac{f}{u_\b}+2\frac{|\n p_\b|^2}{p_\b^2}.$$ We then note that
$$
\n u_\b=\frac{\n \Phi_\b}{p_\b}-\frac{u_\b}{p_\b}\n p_\b=\frac{\n \Phi_\b}{p_\b}-\frac{\Phi_\b}{p_\b^2}\n p_\b.
$$ 
Plugging this expression into \eqref{Eq:10} yields
\begin{align*}
-\Delta \Phi_\b&=fp_\b-2\left(\frac{\n \Phi_\b}{p_\b}-\frac{\Phi_\b}{p_\b^2}\n p_\b\right)\cdot \n p_\b \\
&=\Phi_\b \frac{f}{u_\b}+2\Phi_\b\frac{|\n p_\b|^2}{p_\b^2}+ \langle\n \Phi_\b, B\rangle
\\&=V\Phi_\b +\langle \n \Phi_\b, B\rangle.
\end{align*}
Since $\inf_\O \min\{u_\b, p_\b\}>0$ it follows that $\inf_\O \Phi_\b\geq 0$. We also have $V\geq 0$. By the strong maximum principle we necessarily have 
\begin{equation}\label{Eq:Run}
\min_{\overline \O} \Phi_\b=\min_{\partial \O}\Phi_\b \text{ and this minimum is never reached inside $\O$.}
\end{equation} 
According to the maximum principle of Hopf, if we pick a minimum point $x^*\in \partial \O$ of $\Phi_\b$, there holds
\begin{equation}\label{Eq:Hopf}
\frac{\partial \Phi_\b}{\partial \nu}(x^*)<0.
\end{equation}We fix this point.

We now argue by contradiction and assume that $\H\left(\{\b=0\}\right)>0.$ Given that $\b$ satisfies first order optimality conditions, we can choose such a  minimum point $x^*$ that satisfies
$$x^*\in \{\beta^*=0\}.$$
We finally compute the boundary conditions on $\Phi_\b$. Since 
$$\frac{\partial \Phi_\b}{\partial \nu}=\frac{\partial u_\b}{\partial \nu}p_\b+\frac{\partial p_\b}{\partial \nu}u_\b=-2\beta^*\Phi_\b+j'(u_\b)u_\b,
$$
we get
\begin{equation}\label{CDV}
\frac{\partial \Phi_\b}{\partial \nu}+2\b\Phi_\b=j'(u_\b)u_\b.
\end{equation}
Going back to our minimum point $x^*\in \{\b=0\}$ and to \eqref{Eq:Hopf} we should have $\frac{\partial \Phi_\b}{\partial \nu}(x^*)\leq 0$ on the one hand, and $ \frac{\partial \Phi_\b}{\partial \nu}=-2\b\Phi_\b+j'(u_\b)u_\b=j'(u_\b)u_\b>0$ on the other hand. The last condition comes from the fact that $j$ satisfies \eqref{Eq:Hypj}. This is a contradiction, and the conclusion follows.

\paragraph{Proof of $\boldsymbol{(ii)}$.} Let us use the same notation as in the proof of Theorem~\ref{Th:BgbgBound} in section~\ref{sec:proofTh:BgbgBound}. Let $\beta^*$ denote a solution of Problem~\eqref{Eq:MinBound}, whose existence is guaranteed by Lemma~\ref{Le:Exist}.
We use the Hilbert basis $\{\phi_k\}_{k\in \N}$ of $L^2(\partial\O)$ given by \eqref{def:phik}, associated to the sequence of eigenvalues $\{\sigma_k\}_{k\in\N}$.
According to the first item of Theorem~\ref{Th:RelaxBound}, one has $\H\left(\{0<\beta^*<1\}\right)>0$.

Assume the existence of $C>0$ such that 
$$
j''(u)\leq -Cj'(0)=-C\sup_{u\in [0,U_0(f)]} j'(u).
$$

Let us argue by contradiction, assuming that $\H\left(\{\beta^*=1\}\right)=0$.
Let $h\in L^\infty(\O)$ be such that $\int_\O h=0$. It follows from Footnote~\ref{footnotehAdm}  that $h$ is an admissible perturbation of $\beta^*$ in $\mathcal{B}(\partial\O)$. 
The function $-hu_\b\in L^2(\partial \O)$ expands as
$$
-hu_\b=\sum_{k=0}^\infty \alpha_k(h)\phi_k
\quad\text{with}\quad  
\alpha_k(h)=\int_{\partial \O}(-hu_\b)\phi_k
$$
for every $k\in \N$. Then, we also have 
$$ 
\dot u_\b=\sum_{k=0}^\infty\frac{\alpha_k(h)}{\sigma_k}\phi_k.
$$
Recall that, according to \eqref{Eq:Opt2}, we have
\begin{eqnarray*}
\ddot{\mathcal J}(\b)[h,h]&=& -2\int_{\partial \O}h\dot u_\b p_\b+\int_{\partial \O}j''(u_\b)\left(\dot u_\b\right)^2\\
&=& -2\int_{\partial \O}\Psi_\b hu_\b\dot u_\b +\int_{\partial \O}j''(u_\b)\left(\dot u_\b\right)^2
\end{eqnarray*}
where $\Psi_\b:=\frac{p_\b}{u_\b}$. Observe first that $p_\b \leq  z_\b\sup_{u\in [0,U_0]}j'(u)$ where $z_\b$ is the unique solution to 
$$
\begin{cases}
-\Delta z_\b=0&\text{ in }\O, \\
\frac{\partial z_\b}{\partial \nu}+\b z_\b=1 &\text{ on }\partial\O.
\end{cases}
$$ 
Indeed, the function $P:= p_\b-z_\b  \sup_{u\in [0,U_0]}j'(u)$ is harmonic, and therefore reaches its maximal value on $\partial \O$. Furthermore, according to the Hopf maximum principle, $\partial_\nu P>0$ at this point, which yields easily that $P(\cdot)<0$ on $\partial\O$.

We isolate the following result, which follows from exactly the same arguments as in Lemma \ref{Le:RegRobin}. 
\begin{lemma}\label{lem:0823}
Let $\O$ be a bounded open set of $\R^n$ such that $\partial\O$ is $\mathscr{C}^2$. For $\beta\in\mathcal B( \partial\O)$, we define
$$
\mathcal K(\beta)=\frac{\max_{\overline{\O}} z_\beta}{\min_{\overline{\O}} u_\beta}.
$$
One has
$$
K:=\sup_{\beta\in\mathcal B(\partial\O)}\mathcal K(\beta) <+\infty.
$$
\end{lemma} 
Using the notations of this Lemma, we thus have 
\begin{equation}\left|\Psi_\b \right|\leq \frac{\max p_\b}{\min u_\b}\leq \sup_{[0;U_0(f)]}j'(u) K.\end{equation}

Let us introduce
$$
h_{\alpha_1,\alpha_2}=\frac{\alpha_1\phi_1+\alpha_2\phi_2}{u_\b}
\quad \text{with }\alpha_1=\int_{\partial\O}\frac{\phi_2}{u_\b}\text{ and }\alpha_2=-\int_{\partial\O}\frac{\phi_1}{u_\b},
$$
so that $\int_\O h_{\alpha_1,\alpha_2}=0$. Recall that $h_{\alpha_1,\alpha_2}$ is admissible since $\H\left(\{\beta^*=1\}\right)=\H\left(\{\beta^*=0\}\right)=0$.

From the assumption on $j$ and the Cauchy-Schwarz inequality,  we have
\begin{eqnarray*}
\ddot{\mathcal J}(\b)[h_{\alpha_1,\alpha_2},h_{\alpha_1,\alpha_2}]&\leq & \Vert \Psi_\b\Vert_{\infty}\int_{\partial\O}|h_{\alpha_1,\alpha_2}u_\b||\dot u_\b|-C\sup_{u\in [0,U_0]}j'(u)\int_{\partial\O}(\dot u_\b)^2\\
&\leq & \Vert \Psi_\b\Vert_{\infty}\left(\int_{\partial\O}h_{\alpha_1,\alpha_2}^2u_\b^2\right)^{1/2}\left(\int_{\partial\O} (\dot u_\b)^2\right)^{1/2}-C\sup_{u\in [0,U_0(f)]}j'(u)\int_{\partial\O}(\dot u_\b)^2\\
& \leq & \sup_{u\in [0,U_0(f)]}j'(u)\left[K\left(\alpha_1^2+\alpha_2^2\right)^{1/2}\left(\frac{\alpha_1^2}{\sigma_1^2}+\frac{\alpha_2^2}{\sigma_2^2}\right)^{1/2}-C\left(\frac{\alpha_1^2}{\sigma_1^2}+\frac{\alpha_2^2}{\sigma_2^2}\right)\right]\\
& \leq & \sup_{u\in [0,U_0(f)]}j'(u)\left(\frac{\alpha_1^2}{\sigma_1^2}+\frac{\alpha_2^2}{\sigma_2^2}\right)^{1/2}\left[K\left(\alpha_1^2+\alpha_2^2\right)^{1/2}-C\left(\frac{\alpha_1^2}{\sigma_1^2}+\frac{\alpha_2^2}{\sigma_2^2}\right)^{\frac12}\right]\\
& \leq & \sup_{u\in [0,U_0(f)]}j'(u)\left(\frac{\alpha_1^2}{\sigma_1^2}+\frac{\alpha_2^2}{\sigma_2^2}\right)^{1/2}\left[K\left(\alpha_1^2+\alpha_2^2\right)^{1/2}-C\left(\frac{\alpha_1^2+\alpha_2^2}{\sigma_2^2}\right)^{\frac12}\right]\\
&\leq & \sup_{u\in [0,U_0(f)]}j'(u)\left(\alpha_1^2+\alpha_2^2\right)^{1/2}\left(\frac{\alpha_1^2}{\sigma_1^2}+\frac{\alpha_2^2}{\sigma_2^2}\right)^{1/2}\left(K-\frac{C}{\sigma_2}\right).
\end{eqnarray*}
 Furthermore, according to the Courant-Fisher principle, one has
\begin{eqnarray*}
\sigma_2&=&\min_{\substack{E_2\subset W^{1,2}(\O)\\ \text{subspace of dim. 2}}}\max_{\substack{v\in E_2\\ v\neq 0}}\frac{\int_\O |\nabla v^2|+\int_{\partial\O}\b v^2}{\int_{\O}v^2}\\
&\leq &\min_{\substack{E_2\subset W^{1,2}(\O)\\ \text{subspace of dim. 2}}}\max_{\substack{v\in E_2\\ v\neq 0}}\frac{\int_\O |\nabla v^2|+\int_{\partial\O}v^2}{\int_{\O}v^2}=\Lambda_2(\O)>0,
\end{eqnarray*}
where $\Lambda_2(\O)$ denotes the second Steklov eigenvalue of the domain:
$$
\begin{cases}
-\Delta \varphi_k=0&\text{ in }\O, \\
\frac{\partial \varphi_k}{\partial \nu}+ \varphi_k=\Lambda_k(\O)\varphi_k&\text{ on }\partial\O.
\end{cases}
$$ 
Therefore, if $C$ is such that $\Lambda_2(\O)K<C$, one has $\ddot{\mathcal J}(\b)[h_{\alpha_1,\alpha_2},h_{\alpha_1,\alpha_2}]<0$ and thus
$$
{\mathcal J}(\b+\e h_{\alpha_1,\alpha_2})-{\mathcal J}(\b)=\frac{\e^2}{2} \ddot{\mathcal J}(\b)[h_{\alpha_1,\alpha_2},h_{\alpha_1,\alpha_2}]+\operatorname{o}(\e^2)<0,
$$
whenever $\e>0$ is chosen small enough. This is in contradiction with the optimality of $\b$ whence the result.

%
\section{Proof of Theorem~\ref{Theo:minmaxNRJ}}
We recall that we work with the energy functional
$$
\mathcal F(\beta):=\int_{\O} fu_\beta.
$$
\begin{itemize}
\item[(i)]\textbf{Proof of $\boldsymbol{(i)}$:} Let $\beta\in\mathcal{B}(\partial\O)$ and $h$ denote an admissible perturbation at $\beta$ (see Footnote~\ref{footnotehAdm} for the definition). By mimicking the computations of Section~\ref{Se:OptBound}, one computes
$$
\dot{\mathcal F}(\beta)[h]=-\int_{\partial \O}h u_\beta ^2\quad \text{and}\quad
\ddot{\mathcal F} (\b)[h,h]= -2\int_{\partial\O}hu_\b\dot u_\b,
$$
where $\dot u_\beta$ solves
$$
\begin{cases}
-\Delta \dot u_\beta=0&\text{ in }\O, 
\\ \frac{\partial \dot u_\beta}{\partial \nu}+\beta \dot u_\beta=-hu_\b &\text{ on }\partial \O,
\end{cases}
$$
Let us use the Hilbert basis $\{\phi_k\}_{k\in \N}$ of $L^2(\partial\O)$ given by \eqref{def:phik}, associated with the sequence of eigenvalues $\{\sigma_k\}_{k\in\N}$. The function $-hu_\beta\in L^2(\partial \O)$ expands as
$$
-hu_\beta=\sum_{k=0}^\infty \alpha_k(h)\phi_k
\quad\text{with}\quad   
\alpha_k(h)=\int_{\partial \O}(-hu_\beta)\phi_k
\text{ for every $k\in \N$ }$$and we also have 
$$ 
\dot u_\beta=\sum_{k=0}^\infty\frac{\alpha_k(h)}{\sigma_k}\phi_k.
$$
As a consequence$$
\ddot{\mathcal F} (\b)[h,h]=2\sum_{k=0}^\infty\frac{\alpha_k(h)^2}{\sigma_k}
$$
and we easily infer that $\mathcal J$ is strictly convex.  

Let $\b$ be a solution of Problem~\eqref{Eq:MaxNRJ}. Let us assume by contradiction that the set $\mathcal I:=\{0<\b<1\}$ has positive measure. Let $\tilde \beta$ denote any element of $\mathcal{B}(\partial\O)$ equal to $\beta$ on $\O\backslash \mathcal I$, such that $\tilde \beta\neq \beta$ a.e. on $\mathcal I$ and $\int_{\mathcal I}\b=\int_{\mathcal I}\tilde \beta$. Then, $\b+\e h$ where $h=\tilde \beta-\b$ is admissible and
$$
\mathcal{F}(\b+\e h)-\mathcal{F}(\b)=\frac{\e^2}{2}\ddot{\mathcal J} (\b)[h,h]+\operatorname{o}(\e^2)>0
$$
 whenever $\e$ is small enough. We have thus reached a contradiction and it follows that $\H (\mathcal I)=0$.

\item[(iii)]\textbf{Proof of $\boldsymbol{(ii)}$} According to the analysis above, the mapping $\mathcal{B}(\partial \O)\ni \beta \mapsto \mathcal{F}(\beta)$ is convex. Since we are dealing with a minimization problem, we get that $\b$ solves Problem~\eqref{Eq:MinNRJ} if, and only if $\dot{\mathcal J}(\beta)[h]\geq 0$ for every admissible perturbation $h$. By using the expression of $\dot{\mathcal J}(\beta)[h]$ obtained previously, it is standard that the first order optimality conditions read as follows: there exists a positive real number $\lambda$ such that 
\begin{enumerate}
\item $\{0<\beta^*<1\}\subset\left\{u_{\beta^*}^2=\lambda\right\}$,
\item $\{\beta^*=1\}\subset \{u_{\beta^*}^2\geq \lambda\}$,
\item $\{\beta^*=0\}\subset \{u_{\beta^*}^2\leq \lambda\}$.
\end{enumerate}
Since these conditions are sufficient and necessary, it follows that $\b$ solves Problem~\eqref{Eq:MinNRJ} if and only if $u_\b$ solves the overdetermined system given by \eqref{Eq:MainRobin} complemented by the optimality conditions above. Therefore, to conclude, it is enough to check that the particular function $\b$ given by \eqref{betaStarCasPart} satisfies the optimality system above.

Observe first that, according to the Hopf maximum principle, the function $v_\O$ reaches its minimal value a.e. on the boundary of $\O$ and therefore, one has $\partial_\nu v_\O<0$ on $\partial\O$, meaning that $\b>0$ on $\partial\O$.
Moreover, because of the assumptions on $V_0$, one has
$$
\b < V_0^\O \frac{\Vert \partial_\nu v_\O\Vert_{L^\infty(\partial\O)}}{\int_{\partial\O}\partial_\nu v_\O}=1\text{ a.e. in }\partial\O\quad \text{and}\quad \int_\O\b =V_0,
$$
so that $\b\in \mathcal{B}(\partial\O)$ and $\{0<\b<1\}=\partial\O$.

We set
$$
\lambda=\frac{-\int_{\partial\O}\partial_\nu v_\O}{V_0}, \quad\text{and}\quad u_\b^\lambda=\lambda+v_\O.
$$
Straightforward computations show that $u_\b^\lambda$ coincide with the solution $u_\b$ of \eqref{Eq:MainRobin} and that 
$$
\{0<\beta^*<1\}=\left\{u_{\beta^*}^2=\lambda\right\}.
$$
The expected conclusion follows.

\item[(iii)]\textbf{Proof of $\boldsymbol{(iii)}$:} Using the same arguments as in $(ii)$, since $\{0<\b_{V_0}<1\}=\partial\O$, one sees that $\b_{V_0}$ solves Problem~\eqref{Eq:MinNRJ} if, and only if there exists $\lambda\geq 0$ such that $u_{\b_{V_0}}$ solves the overdetermined system
$$
\begin{cases}
-\Delta u_{\b_{V_0}}=1 &\text{ in }\O, \\
u_{\b_{V_0}}=\lambda &\text{ on }\partial\O\\ 
\frac{\partial u_{\b_{V_0}}}{\partial \nu}=-\frac{V_0}{|\partial\O|} \lambda &\text{ on }\partial\O.
\end{cases}
$$
Setting $v_{\b_{V_0}}=u_{\b_{V_0}}-\lambda$, this overdetermined system is equivalent with 
$$
\begin{cases}
-\Delta v_{\b_{V_0}}=1 &\text{ in }\O, \\
v_{\b_{V_0}}=0 &\text{ on }\partial\O\\ 
\frac{\partial v_{\b_{V_0}}}{\partial \nu}=c_{V_0} &\text{ on }\partial\O.
\end{cases}\quad \text{with}\quad c_{V_0}=-\frac{V_0}{|\partial\O|} \lambda<0.
$$
We can now apply Serrin's theorem (see \cite{MR333220} for the original proof, \cite{Weinberger1971} for a simpler proof that holds in the case $f\equiv 1$, and \cite{Nitsch2017} for a survey of the proofs of this theorem), which fully characterises such overdetermined elliptic problems: this system has a solution if, and only if, $\O$ is a ball. 
\end{itemize}
This concludes the proof of the Theorem.

\medskip

%
%
%


\appendix
\begin{center}
\fbox{\textsf{\Large Appendix}}
\end{center}
\addcontentsline{toc}{part}{Appendices}
\section{Convergence of $u^\Gamma_\alpha$ towards $v^\Gamma$ as $\alpha\to+\infty$}\label{append_vGamma}
We investigate in this section the asymptotic behaviour of $u^\Gamma_\alpha$ as $\alpha\to+\infty$, in the simple case where $g_0=g_0(x,u)$ does not depend on $u$. With a slight abuse of notation, we write $g_0=g_0(x)$.
\begin{proposition}
Let $\Omega$ be a connected bounded open set of class $\mathscr{C}^1$. Let $\Gamma\subset \partial\Omega$, with $\H(\Gamma)>0$, and let $g_0\in L^2(\Omega)$. Assume there exists $\sigma_0>0$ such that $A \geq \sigma_0 \operatorname{Id}$ a.e. in $\Omega$ in the sense of bilinear forms. The family $(u^\Gamma_{\alpha})_{\alpha>0}$ converges to $v^\Gamma$, weakly in $W^{1,2}(\Omega)$ and strongly in $L^2(\Omega)$.
\end{proposition}
 \begin{proof}
 It should be noted that the regularity assumptions on $\O$ is central as it guarantees the compactness of the trace operator $\operatorname{Tr}:W^{1,2}(\O)\to L^2(\partial \O)$.  We write $\operatorname{Tr}_\Gamma$ for the operator that maps $u$ to $\mathds 1_\Gamma\operatorname{Tr}u.$
 Multiplying the first equation of \eqref{Eq:MainIntro} by $u^\Gamma_\alpha$ and integrating by parts gives
 $$\sigma_0\int_\O \left|\n u_\alpha^\Gamma\right|^2+\alpha\int_\Gamma \left(u_\alpha^\Gamma\right)^2\leq \Vert g\Vert_{L^2(\O)}\Vert u_\alpha^\Gamma\Vert_{L^2(\O)}.$$ By continuity of the trace operator and by Lemma \ref{Le:Est}, there exists $C_0=C_0(\sigma_0,\Omega)>0$ such that
 \begin{equation}\label{Eq:Arg} C_0\Vert u^\Gamma_\alpha\Vert_{W^{1,2}(\Omega)}^2\leq \sigma_0\int_{\Omega} |\nabla u^\Gamma_\alpha|^2+\alpha \int_\Gamma (u^\Gamma_\alpha)^2\leq \Vert g\Vert_{L^2(\Omega)}\Vert u^\Gamma_\alpha\Vert_{W^{1,2}(\Omega)}
\end{equation}
Let $(\alpha_n)_{n\in \N}$ be an increasing sequence of positive number such that $\lim_{n\to +\infty}\alpha_n=+\infty$.
As the family $\left\{u^\Gamma_{\alpha_n}\right\}_{n\in\N}$ is bounded in $W^{1,2}(\Omega)$ by \eqref{Eq:Arg}, the Rellich-Kondrachov theorem, ensures that it converges, up to a subsequence, to a certain $\bar u\in W^{1,2}(\Omega)$ weakly in $W^{1,2}(\Omega)$ and strongly in $L^2(\Omega)$. With a slight abuse of notation, this subsequence is still written $\left\{u^\Gamma_{\alpha_n}\right\}_{n\in\N}$.
Since the trace operator is compact, the sequence $\left\{\operatorname{Tr}_\Gamma u^\Gamma_{\alpha_n}\right\}_{n\in\N}$ converges to $\operatorname{Tr}_\Gamma \bar u$ in $L^2(\Gamma)$. As the sequence $\left\{\alpha_n \int_\Gamma (u^\Gamma_{\alpha_n})^2\right\}_{n\in\N}$ is bounded a we must have $\operatorname{Tr}_\Gamma \bar u=0$ in $L^2(\Gamma)$.

Let us introduce the space $W^{1,2}_\Gamma(\Omega)$ as the subspace of functions $\varphi$ in $W^{1,2}(\Omega)$ whose trace vanishes on $\Gamma$. Recall that $u^\Gamma_\alpha$ solves the minimization problem 
$$
\min_{u\in W^{1,2}(\Omega)}\mathcal{F}_\alpha(u)\quad \text{where }\quad \mathcal{F}_\alpha(u)=\frac12  \int_{\Omega} \langle A\nabla u,\nabla u\rangle+\frac{\alpha}{2} \int_\Gamma u^2-\int_\Omega fu.
$$  
By minimality, one has, for any $n\in \N$,
$$
\min_{u\in W^{1,2}(\Omega)}\mathcal{F}_{\alpha_n}(u)=\mathcal{F}_{\alpha_n}(u^\Gamma_{\alpha_n})\leq \min_{u\in W^{1,2}_\Gamma(\Omega)}\mathcal{F}_{\alpha_n}(u)=\min_{u\in W^{1,2}_\Gamma(\Omega)}\mathcal{F}_{0}(u)=\mathcal{F}_{0}(v^\Gamma),
$$  
where $v^\Gamma$ solves Problem~\eqref{Eq:MainIntro2} with mixed Dirichlet-Neumann boundary conditions.
Furthermore, as $A$ is uniformly positive in the sense of bilinear forms, the map  $W^{1,2}(\Omega)\ni u\mapsto \int_{\Omega}\langle A\n u\,, \n u\rangle$ is convex, and, so, weakly lower semi-continuous. Hence, we have
\begin{eqnarray*}
\liminf_{n\to +\infty}\frac12  \int_{\Omega}\langle A\n u_{\alpha_n}^\Gamma\,, \n  u_{\alpha_n}^\Gamma\rangle+\frac{\alpha_n}{2} \int_\Gamma (u_{\alpha_n}^\Gamma)^2-\int_\Omega fu_{\alpha_n}^\Gamma & \geq & \frac12  \int_{\Omega}\langle A\nabla \bar u\,, \n \bar u\rangle-\int_\Omega f\bar u\\
&\geq & \min_{u\in W^{1,2}_\Gamma(\Omega)}\mathcal{F}_{0}(u).
\end{eqnarray*}
Combining both inequalities above, it follows that
$$
\min_{u\in W^{1,2}_\Gamma (\Omega)}\frac12  \int_{\Omega}\sigma |\nabla u|^2-\int_\Omega fu=\frac12  \int_{\Omega}\sigma |\nabla \bar u|^2-\int_\Omega f\bar u,
$$
and by uniqueness of the minimiser of this last problem, we obtain that $\bar u=v^\Gamma$. Thus,  the sequence $\left\{u^\Gamma_{\alpha_n}\right\}_{n\in\N}$ has a unique closure point $v^\Gamma$. It follows that the entire sequence $\left\{u^\Gamma_{\alpha_n}\right\}_{n\in\N}$ converges to $v^\Gamma$, weakly in $W^{1,2}(\Omega)$ and strongly in $L^2(\Omega)$.
 \end{proof}
 \section{Proof of results stated in section \ref{Se:Prelim}}\label{Ap:Prelim}
 
\begin{proof}[Proof of Lemma \ref{Le:Est}]
Of course, it suffices to prove that there exists $C>0$ such that

\begin{equation}\label{Eq:Est2}\forall \beta\in \mathcal B(\partial \O)\,, \forall v\in W^{1,2}(\O)\,,  C\Vert v\Vert_{L^2(\O)}^2\leq \int_\O |\nabla v|^2 +\int_{\partial\O}\beta v^2.
\end{equation}
To prove \eqref{Eq:Est2} we argue by contradiction: should no such constant $C$ exist, there exists a sequence $\{v_n\,, \beta_n\}_{n\in \N}\in \left(W^{1,2}(\O)\times \mathcal B(\partial \O)\right)^\N$ such that$$
\Vert v_n\Vert_{L^2(\O)}^2=1\quad \text{and}\quad \int_\O |\nabla v_n|^2 +\int_{\partial\O} \beta_nv_n^2<\frac{1}{n}
$$
By the Rellich-Kondrachov theorem, there exists $\bar v\in W^{1,2}(\O)$ such that $\{v_n\}_{n\in\N}$ converges, up to a subsequence, to $\bar v$, weakly in $W^{1,2}(\O)$ and strongly in $L^2(\O)$. Denoting this subsequence by $\{v_n\}_{n\in\N}$ with a slight abuse of notation, it follows that
$$
\Vert \bar v\Vert_{L^{2}(\O)}^2= 1\quad \text{and}\quad \int_\O |\nabla \bar v|^2\leq \liminf_{n\to +\infty}\int_\O |\nabla v_n|^2=0.
$$ Thus $\overline v$ is a positive constant, say $\overline v_0$. On the other hand, since $\mathcal B(\partial \O)$ is compact for the weak $L^\infty-*$ topology, there exists $\beta \in \mathcal B(\O)$  such that, still up to a subsequence, 
$$\beta_n \underset{n\to \infty}\rightharpoonup \beta.$$
By compactness of the trace operator, $\{v_n\}_{n\in\N}$ converges strongly, in $L^2(\partial \O)$, to $\overline v$. Thus, passing to the limit in 
$$ \int_\O |\nabla v_n|^2 +\int_{\partial\O} \beta_nv_n^2<\frac{1}{n}$$ yields
$$ \int_\O |\nabla \overline v|^2 +\int_{\partial\O} \beta \overline v_n^2=0.$$Since $\int_{\partial \O}\beta=V_0$, this is impossible as $\overline v$ is a positive constant.

\end{proof}
\begin{proof}[Proof of Lemma \ref{Le:RegRobin}]
The proof of this result relies results for Neumann boundary conditions: consider the problem
\begin{equation}\label{Eq:RegNeuLp}
\begin{cases}
-\Delta u_{f,g}^N=f&\text{ in } \O\,,
\\ \frac{\partial u_{f,g}^N}{\partial \nu}=g&\text{ on }\partial \O.
\end{cases}
\end{equation}
The following regularity holds from, for example,  \cite[Theorem 4.4]{Simader1992}:
assume $\O$ has a $\mathscr C^2$ boundary.
Let $f\in L^q(\O)\,, g\in W^{-\frac1q,q}(\partial \O)$ satisfy the compatibility condition 
$$\int_{\O} f=\int_{\partial \O}g.$$ Then, there exists a $W^{1,q}(\O)$ solution $u_{f,g}$ of \eqref{Eq:RegNeuLp}. Furthermore, for any such solution, 
\begin{equation}\Vert \n u_{f,g}\Vert_{L^q(\O)}\leq C \left(\Vert f\Vert_{L^q(\O)}+\Vert g\Vert_{W^{-\frac1q,q}(\partial \O)}\right).\end{equation}

We turn back to the proof of Lemma \ref{Le:RegRobin}: let us consider, for any $\beta\in \mathcal B(\partial \O)$, the energy functional
\begin{equation}\mathcal E_{\beta,f}:W^{1,2}(\O)\ni u\mapsto \frac12\int_\O |\n u|^2-\int_\O fu+\int_{\partial \O}\beta u^2.\end{equation}
By Lemma \ref{Le:Est} this energy functional is coercive. As a consequence, it admits a minimiser. It is immediate to see that uniqueness holds for Equation \eqref{Eq:MainRobin}. Thus, we have obtained a unique solution $u_\beta\in W^{1,2}(\O)$.

From the Sobolev embeddings $W^{1,2}(\O)\hookrightarrow W^{\frac12,2}(\partial \O)\hookrightarrow L^{2^*_{\partial \O}}(\partial \O)$ where $2^*_{\partial \O}=\frac{2(n-1)}{n-2}$ (if $n\geq 2$, the case $n=1$ being trivial) and the fact that $\beta \in L^\infty(\partial \O)$ we obtain that $u_{\beta}$ solves a Neumann problem with Neumann data $g:=-\beta u_{\beta}\in L^{2^*_{\partial \O}}(\partial \O)$. Furthermore, for the constant $C$ given by Lemma \ref{Le:Est},
$$\Vert g\Vert_{L^{2^*_{\partial \O}}(\partial \O)}\leq \frac1C \Vert f\Vert_{L^2(\O)}.$$
Indeed, we have
\[C\Vert u_\beta\Vert_{W^{1,2}(\O)}^2\leq\int_\O|\n u_\beta|^2+\int_{\partial \O} \beta u_\beta^2\leq \Vert f\Vert_{L^2(\O)}\Vert u_\beta\Vert_{W^{1,2}(\O)},\] and it suffices to invoke the continuity of the trace application, and of the bound $|\beta|\leq 1$.
 Since $\O$ is $\mathscr C^2$, it follows from the regularity for Neumann problems that 
$$u_{\beta}\in W^{1,2^*_{\partial \O}}(\O).$$ We can then bootstrap this argument and obtain successively that 
\begin{equation}\forall k \in \N\,, u_{\beta}\in W^{1,q_k}(\O)\end{equation} where the sequence $\{q_k\}_{k\in \N}$ is defined, by recurrence, as
$$q_{k+1}:=\frac{(n-1)q_k}{n-1-\frac{q_k}2} \text{ if }\frac{q_k}2<n-1\,, q_k+1\text{ else.}$$
The conclusion follows: for any $\beta \in \mathcal B(\partial \O)$, for any $p\in [1;\infty)$, 
$$u_\beta \in W^{1,p}(\O)$$ and furthermore
$$\sup_{\beta \in \mathcal B(\partial \O)}\Vert u_\beta\Vert_{W^{1,p}(\O)}<\infty.$$
To obtain the uniform estimate $$\sup_{\beta \in \mathcal B(\partial \O)}\sup_\O u_\beta<\infty$$ (the symmetric estimate $\inf_{\beta \in \mathcal B(\partial \O)}\inf_\O u_\b>0$ is obtained in the same way) it suffices to take a maximising sequence $\{\beta_k\}_{k\in \N}$ for $\Vert u_\beta\Vert_{L^\infty(\O)}$. Up to a subsequence, $\{\beta_k\}_{k\in \N}$ weakly converges to $\beta \in \mathcal B(\partial \O)$. By Sobolev embeddings, $\{u_{\beta_k}\}_{k\in \N}$ is uniformly bounded in $\mathscr C^{0,\alpha}(\O)$ for some $\alpha>0$. Hence, it converges, strongly in $\mathscr C^0(\O)$, to $u_\beta$, which concludes the proof.

\end{proof}

\begin{proof}[Proof of Lemma \ref{Le:Exist}]
Let us first underline that the set $\mathcal B(\partial \O)$ defined in \eqref{Eq:AdmRobin} endowed with the weak-star topology of $L^\infty(\O)$ is compact. To apply the direct method in the calculus of variations, it suffices to show that all the functionals that define problems \eqref{Eq:MaxBound}-\eqref{Eq:MinBound}-\eqref{Eq:MaxDist}-\eqref{Eq:MinDist} are continuous under this weak $L^\infty-*$ topology. All these maps write as $\int_\O j(u_\beta)$ or $\int_{\partial \O}j(u_\beta)$. Since the functions $u_\beta$ are uniformly bounded from above by Lemma \ref{Le:RegRobin}, the dominated convergence theorem  implies that these functionals are continuous if the map 
$$\mathcal B(\partial \O)\ni\beta\mapsto u_\beta\in W^{1,2}(\O)$$ is continuous for the weak $L^\infty-*$ topology on $\mathcal B(\partial\O)$ and the weak $W^{1,2}$ topology on $W^{1,2}(\O)$. Indeed, it then suffices to invoke the compactness of the embeddings $W^{1,2}(\O)\hookrightarrow L^2(\O)$ and $W^{1,2}(\O)\hookrightarrow L^2(\partial \O)$.

Let us then prove the continuity of $\beta\mapsto u_\beta$ for these weak topologies. Let $\{\beta_n\}_{n\in\N}\in \mathcal{B}(\partial\O)^\N$ be a weakly converging sequence in $\mathcal B(\partial \O)$. Let $\beta \in \mathcal B(\O)$ be such that 
$$\beta_n \underset{n\to \infty}\rightharpoonup \beta.$$ In order to alleviate notations, we define, for any $n\in \N$,  $u_n$ as the solution of \eqref{Eq:MainRobin} associated with $\beta_n$. Our goal is to show
\begin{equation}\label{Eq:Jo}
u_n\underset{n\to \infty}\rightharpoonup u_\beta \text{ in }W^{1,2}(\O).\end{equation}
First of all, multiplying the main equation of \eqref{Eq:MainRobin} by $u_n$ and integrating by parts yields
$$
\int_\O |\nabla u_n|^2 +\int_{\partial\O}\beta_nu_n^2=\int_\O fu_n\leq \Vert f\Vert_{L^2(\partial \O)}\Vert u_n\Vert_{L^2(\O)}.
$$
From Lemma \ref{Le:Est}, it follows that $\{u_n\}_{n\in \N}$ is uniformly bounded in $L^2(\O)$ and, in turn, in $W^{1,2}(\O)$.
%
%
%
%
From the Rellich-Kondrachov theorem there exists $\bar u\in W^{1,2}(\O)$ such that $\{u_n\}_{n\in\N}$ converges, up to a subsequence, to $\bar u$ weakly in $W^{1,2}(\O)$ and strongly in $L^2(\O)$. By the compactness of the trace operator, $\{u_n\}_{n\in \N}$ converges to $\overline u$ strongly in $L^2(\partial \O)$. Passing to the limit in the weak formulation of \eqref{Eq:MainRobin}, we obtain that $\overline u$ is the solution of \eqref{Eq:MainRobin} associated with $\beta$. This concludes the proof.


\end{proof}

 \section{Definition of Steklov eigenvalues and eigenfunctions}\label{sec:appendSteklov}
 Namely, we consider the resolvent operator $T:L^2(\partial \O)\to L^2(\partial \O)$ defined for all $f \in L^2(\partial\O)$ by
$$
T(f)=\left.z_f\right|_{\partial \O}\text{ where $z_f$ is the unique solution of }\begin{cases}-\Delta z_f=0&\text{ in }\O, 
\\ \frac{\partial z_f}{\partial \nu}+\b z_f=f&\text{ on }\partial \O.\end{cases}
$$ 
By compactness of the trace operator and standard regularity estimates, $T$ is a compact operator. It is furthermore self-adjoint since, for any $f,g\in L^2(\partial \O)$ there holds
\begin{eqnarray*}
\int_{\partial \O}T(f)g&=&\int_{\partial \O}z_f\left(\frac{\partial z_g}{\partial \nu}+\b z_g\right)=\int_{\partial \O}\b z_fz_g+\int_\O z_f\Delta z_g-\int_\O z_g\Delta z_f+\int_{\partial \O}z_g\frac{\partial z_f}{\partial \nu}
\\&=&\int_{\partial \O}z_g\left(\frac{\partial z_f}{\partial \nu}+\b z_f\right)=\int_{\partial \O}T(g)f.
\end{eqnarray*}
Finally, $T$ is a positive operator: for any $f\in L^2(\partial \O)$ we have
$$
\int_{\partial \O}T(f)f=\int_{\partial \O}z_f\frac{\partial z_f}{\partial \nu}+\int_{\partial \O}\b z_f^2=\int_\O |\n z_f|^2+\int_{\partial \O}\b z_f^2.
$$
According to the spectral decomposition Theorem, there exists a non-increasing sequence of positive eigenvalues $\{r_k\}_{k\in \N}$ converging to zero and an associated family $\{\phi_k\}_{k\in \N}$ of eigenfunctions satisfying for all $k\in\N$, $T(\phi_k)=r_k \phi_k$. 
Furthermore, the family $\{\phi_k\}_{k\in \N}$ is a Hilbert basis of $L^2(\partial\O)$ and we have 
\begin{equation}
\begin{cases}
-\Delta \phi_k=0&\text{ in }\O, 
\\\frac{\partial \phi_k}{\partial \nu}+\beta_k\phi_k=\frac1{r_k}\phi_k&\text{ on }\partial \O, 
\end{cases}\quad \text{and}\quad \int_{\partial \O}\phi_k\phi_{k'}=\delta_{k,k'}.
\end{equation}
for all $(k,k')\in \N^2$. 
Let us set 
$\sigma_k:=\frac1{r_k}$, we have
\begin{equation}
\sigma_p\xrightarrow[p\to +\infty]{} +\infty\quad \text{and}\quad  \begin{cases}
-\Delta \phi_k=0&\text{ in }\O, 
\\\frac{\partial \phi_k}{\partial \nu}+\beta_k\phi_k=\sigma_k\phi_k&\text{ on }\partial \O, 
\\ \int_{\partial \O}\phi_k^2=1,
\end{cases}
\end{equation}
for all $k\in \N$. Moreover, 
\begin{equation}\forall k,k'\in \N, \int_{\partial \O}\phi_k\phi_{k'}=\delta_{k,k'}.\end{equation} 
Alternatively, we can define, for any $k\in \N$, $\sigma_k$ via the min-max formula
\begin{equation}
\sigma_k:=\min_{\substack{S\text{ subspace of dim }\\ k+1 \text{ of $W^{1,2}(\O)$}}}\max_{v\in S\backslash\{0\}}\frac{\int_\O |\n v|^2+\int_{\partial \O}\b v^2}{\int_{\partial \O}v^2}.
\end{equation}

 \section{Proof of Theorem \ref{Th:BgbgDist}}\label{Ap:Dist}
We omit the subscript in $\mathcal J_{\partial \O}$ and simply write 
$$\mathcal J(\beta):=\int_{ \O} j(u_\beta).$$ 
To prove Theorem \ref{Th:BgbgDist}, we simply need to obtain a lower estimate of the second order derivative of the type given in Proposition \ref{Pr:Rayleigh}. Indeed, the rest of the proof can be adapted verbatim.
We thus need  the  derivatives of the functional under consideration.The first and second order derivatives of $\beta\mapsto u_\beta$ are still denoted by $\dot u_\beta$ and $\ddot u_\beta$.  The equations on $\dot u_\beta$ and $\ddot u_\beta$ remain the same as in the proof of Theorem \ref{Th:BgbgBound}: $\dot u_\beta$ solves \eqref{Eq:DotuBound}, while $\ddot u_\beta$ satisfies \eqref{Eq:DdotuBound}. We now compute the derivatives of $\mathcal J$; they write:
\begin{equation}\label{Eq:SVDist}\dot{\mathcal J}(\beta)[h]=\int_\O \dot u_\beta j'(u_\beta)\text{ and } \ddot{\mathcal J}(\beta)[h,h]=\int_{ \O}\ddot u_\beta j'(u_\beta)+\int_{ \O}\left(\dot u_\beta\right)^2 j''(u_\beta).
\end{equation}
We define the adjoint state $p_\beta$ as the unique  solution in $W^{1,2}(\O)$ of 
\begin{equation}\label{DistEq:Adjoint}
\begin{cases}
-\Delta p_\beta=j'(u_\beta)&\text{ in }\O, 
\\ \frac{\partial p_\beta}{\partial \nu}+\beta p_\beta=0&\text{ on }\partial \O.\end{cases}
\end{equation}
Since $\inf_\O u_\beta>0$ and since $j'>0$ on $\R_+^*$, the maximum principle entails \begin{equation}\label{DistEq:CG}\inf_{\overline \O}p_\beta>0.\end{equation}
If we multiply \eqref{DistEq:Adjoint} by $\dot u_\beta$ and \eqref{Eq:DotuBound} by $p_\beta$, integrating by parts leads to  
\begin{align*}
0&=\int_\O \n \dot u_\beta\cdot \n p_\beta+\int_{\partial \O}\beta p_\beta \dot u_\beta-\int_{ \O}j'(u_\beta)\dot u_\beta
\\0&=\int_\O \n \dot u_\beta\cdot \n p_\beta+\int_{\partial \O}\beta p_\beta \dot u_\beta+\int_{\partial \O} hu_\beta p_\beta 
\end{align*}
so that$$
\int_{ \O} j'(u_\beta)\dot u_\beta=-\int_{\partial \O}hu_\beta p_\beta.$$
This leads to
\begin{equation}\label{DistEq:OptO1}
\dot{\mathcal J}(\beta)[h]=-\int_{\partial \O}h u_\beta p_\beta.
\end{equation}

Similarly, we get
 \begin{equation}\label{DistEq:Opt2}\ddot{\mathcal J}(\beta)[h,h]=-2\int_{\partial \O}h\dot u_\beta p_\beta+\int_{ \O}j''(u_\beta)\left(\dot u_\beta\right)^2.
\end{equation}

We make a proof by contradiction: let  $\beta^*\in \mathcal B(\partial \O)$  be a maximiser such that the set \begin{equation}\label{DistEq:Anormal}\omega^*:=\{0<\beta^*<1\}\end{equation} has positive measure:
$$\H(\omega^*)>0.$$ Thus, for any admissible perturbation $h$ at $\b$ supported in $\omega^*$  we must have \begin{equation}\dot{\mathcal J}(\b)[h]=0.\end{equation}We now prove that there exists an admissible perturbation $h$ supported in $\omega^*$ such that 
\begin{equation}\ddot{\mathcal J}(\b)[h,h]>0.\end{equation}
 
 We aim at obtaining an expression of $\ddot{\mathcal J}$ that is similar to a Rayleigh quotient, since this is the main point of the proof.  We can show the following adaptation of Proposition \ref{Pr:Rayleigh} there exist three constants $A\,, B\,, C$ with $A>0$ such that 
\begin{equation}\label{DistEq:Vivaldi}
\ddot{\mathcal J} (\b)[h,h]\geq A\int_\O  |\n \dot u_\b|^2-B\Vert  \dot u_\b\Vert_{W^{1,2}(\O)}\Vert \dot u_\b\Vert_{L^2(\O)}  -C\int_{\partial \O}\dot u_\b^2.
\end{equation}

To that effect, we set
\begin{equation}{\mathcal W}(\b)[h,h]:=-2\int_{\partial \O} h\dot u_\b p_\b.\end{equation} Using the same computations as in the proof of Proposition \ref{Pr:Rayleigh}, if we set
$$\Psi_\b:=\frac{p_\b}{u_\b},$$ we get
\begin{equation}\mathcal W(\b)[h,h]=\int_{\partial \O}\Psi_\b  \left(\partial_\nu\left({\dot u_\b}^2\right)+2\b \dot u_\b^2\right).\end{equation}We note that $$
\inf_{\overline \O}\Psi_\b>0
$$ 
and that 

\begin{equation}\frac{\partial \Psi_\b}{\partial \nu}=\frac{\partial_\nu p_\b}{u_\b}-\Psi_\b \frac{\partial_\nu u_\b}{u_\b}=-\b \Psi_\b+\b\Psi_\b=0.\end{equation} 
We use identity \eqref{Eq:B} and get \begin{equation}\ddot{\mathcal J} (\b)[h,h]=2\int_\O \Psi_\b |\n \dot u_\b|^2-\int_\O {\dot u_\b^2}\Delta \Psi_\b+\int_{\partial \O}\left(2\b\Psi_\b+j''(u)\right) \dot u_\b^2.
\end{equation}

Since $\b$ and $u_\b$ belong to $L^\infty(\partial \O)$ and since $j\in \mathscr C^2$, there exists a constant $C$ independent of $h$ such that, for any admissible perturbation $h$ 
\begin{equation}
\int_{\partial \O}\left(2\b\Psi_\b+j''(u)\right)\left(\dot u_\b\right)^2\geq -C \int_{\partial \O } \dot u_\b^2.
\end{equation} The rest of the proof follows exactly the same lines: indeed, the rest of the proof of Theorem \ref{Th:BgbgBound} hinges upon the analysis of $\dot u_\b$, not on the fact that the criterion to optimise is distributed. The equation on $\dot u_\b$ remains unchanged, an so does the rest of the analysis.
 \section{Proof of Theorem \ref{Th:RelaxDist}}\label{Ap:RelaxDist}
We define
$$\mathcal J(\beta)=\int_{ \O}j(u_\beta).$$
We shall make use of the computations of Appendix \ref{Ap:Dist}.

We recall that the first order derivative of the criterion is given in \eqref{Eq:SVDist}. We set, for any $\beta\in \mathcal B(\partial \O)$, 
$$\Phi_\beta:=u_\beta p_\beta$$ where $p_\beta$ is the solution of \eqref{DistEq:Adjoint}. Therefore, for any admissible perturbation $h$ at a given $\beta$, there holds
\begin{equation}\dot{\mathcal J}(\beta)[h]=-\int_{\partial \O} h \Phi_\beta.\end{equation} 
Let $\beta^*$ be  a solution of \eqref{Eq:MinDist}. Since we must have 
$$-\dot{\mathcal J}(\beta^*)[h]=\int_{\partial \O}\Phi_{\beta^*}h\leq 0$$ for any admissible perturbation $h$ at $\beta^*$, there exists a real number $\lambda$ (necessarily positive as $\Phi_\beta>0$) such that 
\begin{enumerate}
\item $\{0<\beta^*<1\}\subset\left\{\Phi_{\beta^*}=\lambda\right\}$,
\item $\{\beta^*=1\}\subset \{\Phi_{\beta^*}\geq \lambda\}$,
\item $\{\beta^*=0\}\subset \{\Phi_{\beta^*}\leq \lambda\}$.
\end{enumerate}
As in the proof of Theorem \ref{Th:RelaxBound}, we show that these conditions imply

\begin{equation}\label{AnDist:Yellow}
\H\left(\{\beta^*=0\}\right)=0.
\end{equation} The required conclusion then follows.

To prove \eqref{AnDist:Yellow}, let us first observe that 
\begin{equation}\label{AnDist:10}
\n\Phi_{\beta^*}=u_\b\n p_\b+p_\b\n u_\b, -\Delta \Phi_\b=-p_\b\Delta u_\b-u_\b \Delta p_\b-2 \n u_\b\cdot \n p_\b.
\end{equation}
First, we set $$B:=-2\frac{\n p_\b}{p_\b}$$ and $$V:=\frac{f}{u_\b}+\frac{j'(u_\b)}{p_\b}+2\frac{|\n p_\b|^2}{p_\b^2}.$$ Second, observe that
$$
\n u_\b=\frac{\n \Phi_\b}{p_\b}-\frac{u_\b}{p_\b}\n p_\b=\frac{\n \Phi_\b}{p_\b}-\frac{\Phi_\b}{p_\b^2}\n p_\b.
$$ 
Plugging this expression into \eqref{AnDist:10} yields
\begin{align*}
-\Delta \Phi_\b&=fp_\b+j'(u_\b)u_\b-2\left(\frac{\n \Phi_\b}{p_\b}-\frac{\Phi_\b}{p_\b^2}\n p_\b\right)\cdot \n p_\b \\
&=\Phi_\b \left(\frac{f}{u_\b}+\frac{j'(u_\b)}{p_\b}\right)+2\Phi_\b\frac{|\n p_\b|^2}{p_\b^2}+ \langle\n \Phi_\b, B\rangle
\\&=V\Phi_\b +\langle \n \Phi_\b, B\rangle.
\end{align*}
Since $\inf_\O \min\{u_\b, p_\b\}>0$ it follows that $\inf_\O \Phi_\b\geq 0$. We also have $V\geq 0$. We ccan then follow the proof of Theorem \ref{Th:RelaxBound} verbatim.
 \section{Proof of Theorem \ref{Theo:NL}}\label{Ap:NL}
 We recall that
$$\mathcal R(\beta):=\int_{ \O} j(y_\beta).$$ 
To prove Theorem \ref{Theo:NL}, we  need a lower estimate of the second order derivative of the type given in Proposition \ref{Pr:Rayleigh}; in a second step, we will need a set of eigenfunctions different from the ones used in the proofs of Theorem \ref{Th:BgbgBound}. Although not straightforward, we show how this can be done.

We start by computing the derivatives of the criterion. The first and second order derivatives of $\beta\mapsto y_\beta$ are denoted by $\dot y_\beta$ and $\ddot y_\beta$. By Assumptions \eqref{Eq:Hg}-\eqref{Eq:HStab}, these derivatives exist. Furthermore, $\dot y_\beta$ satisfies
\begin{equation}\label{NLEq:DotyBound}\begin{cases}
-\Delta \dot y_\beta=\frac{\partial g}{\partial u}(x,y_\beta)\dot y_\beta&\text{ in }\O, 
\\ \frac{\partial \dot y_\beta}{\partial \nu}+\beta \dot y_\beta=-hy_\beta &\text{ on }\partial \O,
\end{cases}
\end{equation}
while $\ddot y_\beta$ satisfies
\begin{equation}\label{NLEq:DdotyBound}\begin{cases}
-\Delta \ddot y_\beta=\frac{\partial^2g}{\partial u^2}(x,y_\beta)\left(\dot y_\beta\right)^2+\frac{\partial g}{\partial u}(x,y_\beta)\ddot y_\beta&\text{ in }\O, 
\\ \frac{\partial \ddot y_\beta}{\partial \nu}+\beta \ddot y_\beta=-2h\dot y_\beta &\text{ on }\partial \O.
\end{cases}
\end{equation}

 We now compute the derivatives of $\mathcal R$; they write:
\begin{equation}\label{NLEq:SVDist}\dot{\mathcal R}(\beta)[h]=\int_\O \dot y_\beta j'(y_\beta)\text{ and } \ddot{\mathcal R}(\beta)[h,h]=\int_{ \O}\ddot y_\beta j'(y_\beta)+\int_{ \O}\left(\dot y_\beta\right)^2 j''(y_\beta).
\end{equation}
We define the adjoint state $p_\beta$ as the unique  solution in $W^{1,2}(\O)$ of 
\begin{equation}\label{NLEq:Adjoint}
\begin{cases}
-\Delta p_\beta=\frac{\partial g}{\partial u}(x,y_\beta)p_\beta+j'(y_\beta)&\text{ in }\O, 
\\ \frac{\partial p_\beta}{\partial \nu}+\beta p_\beta=0&\text{ on }\partial \O.\end{cases}
\end{equation}

We need to check that a solution to this equation indeed exists.
\begin{lemma}\label{NLLe:Exist}
There exists a unique solution $p_\beta\in W^{1,2}(\O)$ of \eqref{NLEq:Adjoint}. For any $p\in [1;+\infty)\,, p_\beta\in W^{1,p}(\O)$ and 
$$\inf_{\overline \O}p_\beta>0.$$
\end{lemma}
\begin{proof}[Proof of Lemma \ref{NLLe:Exist}]
By Assumption \eqref{Eq:HStab}, the energy functional 
$$E:W^{1,2}(\O)\ni p\mapsto \frac12\int_\O |\n p|^2-\frac12\int_\O\frac{\partial g}{\partial u}(x,y_\beta) p^2-\int_\O j'(y_\beta)p+\int_{\partial \O}\beta p^2$$ can be bounded from below as 
$$E(p)\geq \mu_\beta \int_\O p^2-\int_\O \frac{\partial g}{\partial u}(x,y_\beta)p.$$ Hence, $E$ is coercive, and so a solution $p_\b$ to \eqref{NLEq:Adjoint} exists. To prove the uniqueness of this solution, we argue by contradiction: if two solutions $p_\beta\,, q_\beta$ exist, then $z_\beta:=p_\beta-q_\beta$ solves
\begin{equation}\begin{cases}
-\Delta z_\beta=\frac{\partial g}{\partial u}(x,y_\beta)z_\beta&\text{ in }\O\,, 
\\ \frac{\partial z_\beta}{\partial \nu}+\beta z_\beta=0&\text{ on }\partial \O.\end{cases}\end{equation} Hence, if $p_\beta\neq q_\beta$, $z_\beta\neq 0$ is an eigenfunction of the operator $-\Delta-\frac{\partial g}{\partial u}(x,y_\beta)$, associated with the eigenvalue 0. However, the lowest eigenvalue of this operator is $\mu_\beta>0$, a contradiction. Uniqueness follows.

The $W^{1,p}$-regularity of $p_\beta$ is a consequence of the same arguments as in Lemma \ref{Le:RegRobin}.

Finally, the positivity of $p_\beta$ is a consequence of the following version of the maximum principle: as $\mu_\beta>0$, should $p_\beta$ not be positive, the negative part $p_\beta^-:=-p_\beta\mathds 1_{p_\beta}$ satisfies
$$\int_\O |\n p_\beta^-|^2-\int_\O\frac{\partial g}{\partial u}(x,y_\beta)(p_\beta^-)^2\leq -\int_\O j'(y_\beta)p_\beta^-\leq 0.$$ By the variational formulation \eqref{Eq:DefMu} of $\mu_\beta>0$, we necessarily have $p_\beta^-=0$. So we first have $p_\beta^-\geq 0$. It then suffices to apply the classical maximum principle to conclude.
\end{proof}

If we multiply \eqref{NLEq:Adjoint} by $\dot y_\beta$ and \eqref{NLEq:DotyBound} by $p_\beta$, integrating by parts leads to  
\begin{align*}
0&=\int_\O \n \dot y_\beta\cdot \n p_\beta-\int_\O\frac{\partial g}{\partial u}(x,y_\beta) p_\beta\dot y_\beta+ \int_{\partial \O}\beta p_\beta \dot y_\beta-\int_{ \O}j'(y_\beta)\dot y_\beta
\\0&=\int_\O \n \dot y_\beta\cdot \n p_\beta-\int_\O \frac{\partial g}{\partial u}(x,y_\beta)p_\beta\dot y_\beta+\int_{\partial \O}\beta p_\beta \dot y_\beta+\int_{\partial \O} hy_\beta p_\beta 
\end{align*}
so that$$
\int_{ \O} j'(y_\beta)\dot y_\beta=-\int_{\partial \O}hy_\beta p_\beta.$$
This leads to
\begin{equation}\label{NLEq:OptO1}
\dot{\mathcal R}(\beta)[h]=-\int_{\partial \O}h y_\beta p_\beta.
\end{equation}
Similarly, we get
 \begin{equation}\label{NLEq:Opt2}\ddot{\mathcal R}(\beta)[h,h]=-2\int_{\partial \O}h\dot y_\beta p_\beta+\int_\O \frac{\partial^2 g}{\partial u^2}(x,y_\beta)p_\beta\left(\dot y_\beta\right)^2+ \int_{ \O}j''(y_\beta)\left(\dot y_\beta\right)^2.
\end{equation}

We argue by contradiction: let  $\beta^*\in \mathcal B(\partial \O)$  be a maximiser such that the set \begin{equation}\label{NLEq:Anormal}\omega^*:=\{0<\beta^*<1\}\end{equation} has positive measure:
$$\H(\omega^*)>0.$$ Thus, for any admissible perturbation $h$ at $\b$ supported in $\omega^*$  we must have \begin{equation}\dot{\mathcal R}(\b)[h]=0.\end{equation}We now prove that there exists an admissible perturbation $h$ supported in $\omega^*$ such that 
\begin{equation}\ddot{\mathcal R}(\b)[h,h]>0.\end{equation}
 
 We aim at obtaining an expression of $\ddot{\mathcal R}$ that is similar to a Rayleigh quotient, since this is the main point of the proof.   We show the following adaptation of Proposition \ref{Pr:Rayleigh}: there exist $A>0$ and two constants $B\,, C>0$ such that \begin{equation}\label{NLEq:Vivaldi}
\ddot{\mathcal R} (\b)[h,h]\geq A\int_\O  |\n \dot y_\b|^2-B\Vert  \dot y_\b\Vert_{W^{1,2}(\O)}\Vert \dot y_\b\Vert_{L^2(\O)}  -C\int_{\partial \O}\dot y_\b^2.
\end{equation}

To that effect, we set
\begin{equation}{\mathcal W}(\b)[h,h]:=-2\int_{\partial \O} h\dot y_\b p_\b.\end{equation} As before, if we set $$\Psi_\b:=\frac{p_\b}{u_\b},$$ we obtain
\begin{equation}\mathcal W(\b)[h,h]=\int_{\partial \O}\Psi_\b  \left(\partial_\nu\left({\dot y_\b}^2\right)+2\b \dot y_\b^2\right).\end{equation}We note that $$
\inf_{\overline \O}\Psi_\b>0
$$ 
and that 

\begin{equation}\frac{\partial \Psi_\b}{\partial \nu}=\frac{\partial_\nu p_\b}{u_\b}-\Psi_\b \frac{\partial_\nu u_\b}{u_\b}=-\b \Psi_\b+\b\Psi_\b=0.\end{equation} 
From \eqref{Eq:B} we are led to  \begin{equation}\ddot{\mathcal R} (\b)[h,h]=2\int_\O \Psi_\b |\n \dot y_\b|^2-\int_\O {\dot y_\b^2}\Delta \Psi_\b+\int_{\partial \O}\left(2\b\Psi_\b+j''(u)\right) \dot y_\b^2+\int_\O \frac{\partial^2 g}{\partial u^2}(x,y_\beta)p_\beta\left(\dot y_\beta\right)^2.
\end{equation}

Since $\b$ and $u_\b$ belong to $L^\infty(\partial \O)$, since $g\in \mathscr C^2$ in its second variable and since $j\in \mathscr C^2$, there exists a constant $C$ independent of $h$ such that, for any admissible perturbation $h$ 
\begin{align*}
\int_{\partial \O}\left(2\b\Psi_\b+j''(u)\right)\left(\dot y_\b\right)^2+\int_\O \frac{\partial^2 g}{\partial u^2}(x,y_\beta)p_\beta\left(\dot y_\beta\right)^2&\geq -C \left(\int_{\partial \O } \dot y_\b^2+\int_\O\dot y_\beta^2\right)
\\&\geq -C\int_{\partial \O}\dot y_\beta^2-C\Vert \dot y_\beta\Vert_{L^2(\O)}\Vert \dot y_\beta\Vert_{W^{1,2}(\O)}.
\end{align*} The rest of the proof follows exactly the same lines, except that, instead of considering the sequence of eigenelements defined in \eqref{def:phik}, we rather need the following: we consider the eigenvalues $$
0\leq \sigma_0\leq \sigma_1\leq \dots\leq \sigma_k\underset{k\to \infty}\rightarrow +\infty$$ where, for each $k$, the eigenvalue $\sigma_k$ is associated with the eigenfunction $\phi_k$ solution of 
\begin{equation}\label{NLEq:Phik}
\forall k \in \N\,, \begin{cases}-\Delta \phi_k=\frac{\partial g}{\partial u}(x,y_\beta)p_\beta&\text{ in }\O\,, 
\\ \frac{\partial \phi_k}{\partial \nu}+\beta \phi_k=\sigma_k\phi_k&\text{ on }\partial \O
\end{cases}
\text{ and, for any $k\,, k'\in \N\,, $} \int_{\partial \O}\phi_k\phi_{k'}=\delta_{k,k'}.\end{equation}
The one thing that needs to be checked is that these eigenelements are well defined. This is once again a consequence of the stability Assumption \eqref{Eq:HStab}: proceeding as in Appendix \ref{sec:appendSteklov}, it suffices to show that the operator $T:L^2(\partial \O)\ni f\mapsto T(f)$ defined, for any $f\in L^2(\partial \O)$, as

$$
\left.z_f\right|_{(\partial \O)}\text{ where $z_f$ is the unique solution of }\begin{cases}-\Delta z_f-\frac{\partial g}{\partial u}(x,y_\beta)z_f=0&\text{ in }\O, 
\\ \frac{\partial z_f}{\partial \nu}+\beta z_f=f&\text{ on }\partial \O\end{cases}
$$ 
is compact. First, we need to check that $T$ is well-defined. However, this follows from the same arguments as in Lemma \ref{NLLe:Exist}, considering this time the energy functional 
$$E:W^{1,2}(\O)\ni z\mapsto  \frac12\int_\O |\n z|^2-\frac12\int_\O\frac{\partial g}{\partial u}(x,y_\beta) z^2+\int_{\partial \O}\beta z^2-\int_{\partial \O}zf.$$ Second, the compactness is a consequence of standard $W^{1,2}$-estimates. 

We can hence define the eigenelements described in \eqref{NLEq:Phik}. The rest of the proof is adapted verbatim.

\section*{Acknowledgments}
The authors would like to warmly thank Dorin Bucur for fruitful discussions on shape optimization problems involving solutions of PDEs with Robin boundary conditions.

\nocite{}
\bibliographystyle{abbrv}
\bibliography{Biblio-Robin}

\textsc{Idriss Mazari}
\\CEREMADE, UMR CNRS 7534, Universit\'e Paris-Dauphine, Universit\'e PSL, Place du Mar\'echal De Lattre De Tassigny, 75775 Paris cedex 16, France, \\\texttt{mazari@ceremade.dauphine.fr}.

\textsc{Yannick Privat}
\\IRMA, Universit\'e de Strasbourg, CNRS UMR 7501, Inria, 7 rue Ren\'e Descartes, 67084 Strasbourg, France, 
\\{Institut Universitaire de France (IUF),} 
\\ {\tt yannick.privat@unistra.fr}.

\end{document}